\documentclass[12pt,reqno]{amsart}
\usepackage{amssymb}
\usepackage[curve,matrix,arrow]{xy}
\newtheorem{thm}{Theorem}[section]
\newtheorem{lemma}[thm]{Lemma}
\newtheorem{cor}[thm]{Corollary}
\newtheorem{prop}[thm]{Proposition}

\newtheorem{defi}[thm]{Definition}
\newtheorem{notation}[thm]{Notation}
\theoremstyle{definition}
\newtheorem{rem}[thm]{Remark}

\def\CS{{\mathcal{S}}}

\def\bZ{{\mathbb Z}}
\def\bF{{\mathbb F}}
\def\bK{{\mathbb K}}

\def\bV{{\mathbb V}}

\def\Id{\operatorname{Id}\nolimits}
\def\res{\operatorname{Res}\nolimits}

\def\Hom{\operatorname{Hom}\nolimits}
\def\End{\operatorname{End}\nolimits}

\def\proj{\operatorname{(proj)}\nolimits}

\def\Det{\operatorname{Det}\nolimits}
\def\stmod{\operatorname{stmod}\nolimits}
\def\SL{\operatorname{SL}\nolimits}
\def\GL{\operatorname{GL}\nolimits}

\def\SP#1#2{\operatorname{Sp}\nolimits}

\def\bV{\mathbb V}

\def\bZ{{\mathbb Z}}
\def\bfq{{\mathbb F}_q}
\def\whg{{\widehat G}}
\def\whn{{\widehat N}}
\def\whl{{\widehat L}}
\def\whh{{\widehat H}}
\def\whs{{\widehat S}}
\def\whq{{\widehat Q}}

\def\wtl{{\widetilde L}}
\def\wts{{\widetilde S}}
\def\CL{\mathcal L}
\def\whcl{{\widehat \CL}}
\def\gcd{\operatorname{gcd}}

\newcommand{\ls}[2]{{}^{{#1}}\!{{#2}}}

\title[Endotrivial modules for $\GL_n$]{Endotrivial modules 
for the general linear group in a nondefining 
characteristic}

\author{\sc Jon F. Carlson}
\address
{Department of Mathematics\\ University of Georgia \\
Athens\\ GA~30602, USA}
\thanks{Research of the first author was supported in part by NSF
grant DMS-1001102}
\email{jfc@math.uga.edu}
\author{\sc Nadia Mazza}
\address
{Department of Mathematics\\ University of Lancaster}
\email{nmazza@math.uga.edu}
\author{\sc Daniel K. Nakano}
\address
{Department of Mathematics\\ University of Georgia \\
Athens\\ GA~30602, USA}
\thanks{Research of the third author was supported in part by NSF
grant DMS-1002135}
\email{nakano@math.uga.edu}
\date{\today}

\begin{document}
\maketitle

\begin{abstract}
Suppose that $G$ is a finite group such that 
$\SL(n,q)\subseteq G \subseteq \GL(n,q)$, and
that $Z$ is a central subgroup of $G$. 
Let $T(G/Z)$ be the abelian group of 
equivalence classes of endotrivial $k(G/Z)$-modules, where $k$ is an 
algebraically closed field of characteristic~$p$ not dividing $q$. 
We show that the torsion free rank of $T(G/Z)$ is at most
one, and we
determine $T(G/Z)$ in the case that the Sylow $p$-subgroup of $G$ is
abelian and nontrivial. 
The proofs for the torsion subgroup of $T(G/Z)$ use the  
theory of Young modules for $\GL(n,q)$ and a new method due to 
Balmer for computing the kernel of restrictions in the group
of endotrivial modules. 
\end{abstract}

\section{Introduction}

For a finite group $G$, endotrivial modules form an 
important part of the Picard group of 
self-equivalences of the stable module category, 
$\stmod(kG)$, for the group algebra $kG$. 
In particular, the functor of tensoring with an 
endotrivial module is a self-equivalence of ``Morita type" 
on the stable category as well as on the derived category. 
The isomorphism classes in the stable category of the endotrivial 
modules over $G$ form a group, $T(G)$, under tensor product.  
The endotrivial group $T(G)$ was 
introduced by Dade \cite{D1, D2} who showed that 
$T(G)$ is cyclic in the case that $G$ is an abelian $p$-group. 
For general $p$-groups a classification of the endotrivial
modules was obtained by the first author and Th\'evenaz 
\cite{CT2, CT3, Ca}, building on work of Alperin \cite{A}
and others. 

The problem of computing $T(G)$ for an arbitrary finite group $G$
remains open. Many results for specific groups can be found 
in \cite{CHM, CMN1, CMN2, CMT, LMS, NR, Rob}.  
It is hoped that a complete solution
to calculating $T(G)$ can be reduced to almost-simple groups. 
The proof in \cite{CMT} hints that such a reduction may hold. 
The fact that $T(G)$ is finitely generated can be derived from
the theory of Green correspondents and the classification of 
endotrivial modules over $p$-groups. 

This paper is part of a series of efforts to classify endotrivial
modules for almost simple groups. 
In \cite{CMN1}, the authors gave a classification 
of the endotrivial modules for the finite
groups of Lie type where $p$ is the defining 
characteristic.  Here, we consider the 
case of finite groups of Lie type in
nondefining characteristics.  We prove that 
the torsion free part of the group of 
endotrivial modules has either rank 
zero or one if the group has underlying root 
system of type $A_{n}$, and we characterize the 
torsion part of the group provided the Sylow 
$p$-subgroup is abelian. The main results are the following.
For notation, let $k$ be an algebraically closed field of 
characteristic $p>0$ with $\gcd(p,q)=1$. 

First we present the case when the Sylow $p$-subgroup is not cyclic,
which is by far more difficult than when the Sylow $p$-subgroup is
cyclic. 

\begin{thm} \label{thm:noncyclic}
Suppose that $\SL(n,q) \subseteq G \subseteq \GL(n,q)$ 
and that $Z \subseteq Z(G)$.
Assume that $p$ divides the order of $G$ 
and that the Sylow $p$-subgroup
of $G$ is abelian but not cyclic. 
Then 
$$T(G/Z) \cong \bZ \oplus X(G/Z),$$ 
where the
torsion free part is generated by the 
class of $\Omega(k)$. Here $X(G/Z)$
is the group of isomorphism classes 
of $k(G/Z)$-modules of dimension one. 
\end{thm}

An interesting feature of the proof of this theorem
is the first time utilization of a powerful method 
recently introduced by Balmer \cite{Bal}. If $H$
is a subgroup of $G$ that contains a Sylow $p$-subgroup of $G$, 
then Balmer's method finds the kernel of the restriction
map $T(G) \to T(H)$ in combinatorial terms. Applying the method 
requires detailed information about the $H$-$H$ double cosets in 
$G$ and the normalizers all of the $p$-subgroups of $G$. This 
makes the applications difficult except in some special cases.

For the sake of completion, we present also the case that the 
Sylow $p$-subgroup is cyclic. The proof in this case is accomplished
by straightforward applications of known results.  

\begin{thm} \label{thm:cyclic}
Suppose that $\SL(n,q) \subseteq G \subseteq \GL(n,q)$ and that 
$Z \subseteq Z(G)$. Assume that the Sylow $p$-subgroup $S$ 
of $G$ is cyclic and let $N = N_G(S)$. 
Then $T(G/Z) \cong T(\whn)$ where $\whn = N_{G/Z}(\whs)$ and 
$\whs$ is a Sylow $p$-subgroup of $G/Z$. 
Moreover, $T(\whn)$ is the middle term of a not necessarily split
extension 
\begin{equation} \label{eq:sequence}
\xymatrix{
1 \ar[r] & X(\whn) \ar[r] & T(\whn) \ar[r] & T(\whs) \ar[r] & 0
}
\end{equation} 
where $X(\whn)\cong N/(Z[N,N])$ is the 
group of isomorphism classes of
$k\whn$-modules of dimension one.
Let $D = \Det(G) \cong G/\SL(n,q)$ and let 
$d = \vert D \vert$. In the case that $Z = \{1\}$ we have the 
following. 
\begin{itemize}
\item[(a)] If $p=2$ then $n=1$, and $T(G) \cong D/S$.
\item[(b)] Suppose that $p>2$ divides $q-1$. If $p$ divides $d$,
then $n = 1$ and $T(G) \cong \bZ/a\bZ \oplus \bZ/2\bZ$,
where $d = ap^t$ for $a$ relatively prime to $p$.
\item[(c)] If $p>2$ divides $q-1$ and $p$ does not divide $d$, 
then there are two possibilities:
\begin{itemize}
\item[(i)] assuming that~2 does not divide $(q-1)/d$, then 
$T(G) \cong \bZ/d\bZ \oplus \bZ/4\bZ$. 
\item[(ii)] assuming that~2 divides $(q-1)/d$, then $
T(G) \cong \bZ/d\bZ \oplus \bZ/4\bZ \oplus \bZ/2\bZ$. 
\end{itemize}
\item[(d)] Suppose that $p$ does not divide $q-1$. Let $e$ be the least 
integer such that $p$ divides $q^e-1$. Then $n = e+f$ for some 
$f$ with $0 \leq f < e$. Let $m = (q-1)/d$ and $\ell = 
\gcd\big(m(q-1),q^e-1\big)/m$. Then we have two possibilities: 
\begin{itemize}
\item[(i)] if $f=0$ then $T(G) \cong \bZ/\ell\bZ \oplus \bZ/2e\bZ$,
\item[(ii)] while if $f >0$, then  
$T(G) \cong \bZ/2e\bZ \oplus \bZ/(q-1)\bZ \oplus \bZ/d\bZ$ 
(except that $T(G) \cong \bZ/2e\bZ \oplus 
\bZ/2\bZ$ if both $f=2$ and $q = 2$).
\end{itemize}
\end{itemize}
\end{thm}

The paper is organized as follows. In 
Sections 2 and 3 of the paper we present some preliminary 
results on endotrivial modules and general linear groups. In 
particular, we prove that the torsion free rank of the group 
of endotrivial modules is never more than one. This is a consequence
of known facts about the conjugations of $p$-subgroups in 
the general linear group. In Section 4, a proof is given for 
Theorem \ref{thm:noncyclic} in the case that $G = \SL(n,q)$ 
and $p$ divides $q-1$. Balmer's characterization of the kernel
of restriction is reviewed in Section 5. In Section 6 notation 
and preliminaries are set up for Section 7 where 
an important case of Theorem \ref{thm:noncyclic} is proved. 
The remaining cases, for $G = \SL(n,q)$ with noncyclic Sylow 
$p$-subgroup, are considered in Sections 8 and 9.  An essential 
feature of the proofs is the use of the theory of Young modules
for the general linear group.  In the last section, the 
results are assembled and the proofs of 
Theorems~\ref{thm:noncyclic} and \ref{thm:cyclic} are presented.

\vskip.2in 
\noindent{\em Acknowledgements:} 
The authors would like to thank Paul Balmer for many helpful 
conversations and for urging us to look closely at his papers on 
the kernel of restriction. The second author acknowledges the 
Department of Mathematics at the University of Georgia for its support and
hospitality during the initial stages of this project. 
The authors of this paper drew a lot of inspiration from experimental
calculations using the computer algebra system Magma \cite{BoCa}, 
though none of it is actually used in the proofs of this paper. 
We are grateful to Caroline Lassueur for pointing out an error
in an earlier version of the paper. 
Finally, we wish to thank the referee for many helpful suggestions. 
\section{Preliminaries}

Throughout this paper, let $k$ denote an algebraically
closed field of prime characteristic~$p$. We assume 
that all modules are finitely generated.
If $A$ and $B$ are $kG$-modules for $G$ a finite group, 
we write $A \cong B \oplus \proj$ to mean that $A$ is 
isomorphic to the direct sum of $B$ with some 
projective module. We write $k$ for the trivial $kG$-module.  
Unless otherwise specified, the symbol $\otimes$ 
is the tensor product $\otimes_k$ of the underlying vector spaces, 
and in case of $kG$-modules, this is a $kG$-module 
with $G$ acting via the diagonal action on the factors.

We start by reviewing some basics of the study of endotrivial 
modules and set some notation. If $M$ is a $kG$-module, 
and $\varphi: Q \rightarrow M$
its projective cover, then we let $\Omega^1(M)$, or simply $\Omega(M)$,
denote the kernel of~$\varphi$ (called the first syzygy of~$M$).
Likewise, if $\vartheta: M \rightarrow Q$ is the injective
hull of $M$ (recall that $kG$ is a self-injective ring so $Q$ is also 
projective), then $\Omega^{-1}(M)$ denotes the cokernel of $\vartheta$.
Inductively, we set $\Omega^n(M) = \Omega(\Omega^{n-1}(M))$ 
and $\Omega^{-n}(M) = \Omega^{-1}(\Omega^{-n+1}(M))$ for 
all integers $n > 1$.

Assume that $G$ has order divisible by~$p$.
A $kG$-module $M$ is {\em endotrivial} provided its 
endomorphism algebra $\End_k(M)$ is isomorphic 
(as a $kG$-module) to the direct sum of the
trivial module $k$ and a projective $kG$-module.
Recall that $\Hom_k(M, N) \cong M^*\otimes N$ 
where $M^*$ denotes the $k$-dual $\Hom_k(M,k)$ of~$M$. 
In other words, a $kG$-module $M$ is endotrivial 
if and only if $\Hom_k(M, M) \cong M^*\otimes M\cong k\oplus\proj$,

Any endotrivial module $M$ splits as the direct sum
$M_\diamond\oplus\proj$
for an indecomposable endotrivial $kG$-module $M_\diamond$, 
which is unique up to isomorphism.  We define an equivalence relation
$$
M\sim N~\Longleftrightarrow~M_\diamond\cong N_\diamond
$$
on the class of endotrivial $kG$-modules, and let $T(G)$ be the set of
equivalence classes. Every equivalence class contains a unique
indecomposable module up to isomorphism.  The tensor product induces an
abelian group structure on the set $T(G)$ as in $[M]+[N]=[M\otimes N].$ 
The zero element of $T(G)$ is the class~$[k]$ of the 
trivial module, consisting of all modules of the 
form $k\oplus\proj$. The inverse of the class of 
a module $M$ is the class of the dual module~$M^*$.

The group $T(G)$ is called the {\em group of endotrivial $kG$-modules}.
It is known to be a finitely generated abelian group.
In particular, the torsion subgroup $TT(G)$ of $T(G)$ is finite.
We define $TF(G)=T(G)/TT(G)$.
Thus, determining the structure of $T(G)$ can be split into two tasks:
(i) compute $TF(G)$ and (ii) determine $TT(G)$.

The rank of the free abelian group $TF(G)$ is called the 
torsion-free rank of $T(G)$. It can be computed from a 
formula given in~\cite{CMN1}, which extends work 
of Alperin~\cite{A}, and solely depends on the 
$p$-subgroup structure of $G$. 
Of particular interest is the $p$-rank of $G$ which is the rank of the largest 
elementary abelian $p$-subgroup of $G$ (i.e., the 
logarithm to base $p$ of the order of such subgroup).
If the $p$-rank of $G$ is one, then $TF(G)=\{0\}$. 
For an elementary abelian $p$-group $E$ of rank at least~$2$, 
Dade~\cite{D1,D2} proved that $T(E)\cong\bZ$, 
and is generated by the class
$[\Omega(k)]$. Now let $n_G$ be the number of conjugacy classes of
maximal elementary abelian $p$-subgroups 
of $G$ of order $p^2$. These are maximal
in the sense that they are not contained 
in a larger elementary abelian
$p$-subgroup.  

\begin{thm}\label{thm:rank} (\cite{CMN1}, Theorem 3.1)
Assume that the $p$-rank of $G$ is at least~$2$.
The torsion-free rank of $T(G)$ is equal to the number~$n_G$ defined
above if $G$ has $p$-rank $2$, and is equal to $n_G+1$ if $G$ has rank
at least $3$. 
\end{thm}

In~\cite{CT3}, it is shown that if $G$ is a $p$-group, then $TT(G)$ is
trivial unless $G$ is cyclic, dihedral, generalized quaternion or
semi-dihedral. In many cases, the
subgroup $TT(G)$ is determined by the kernel of the restriction
$T(G) \to T(S)$ where $S$ is a Sylow $p$-subgroup of $G$. 
We introduce the following term for the convenience of
stating some of the results.  

\begin{defi}
We say that a $kG$-module $M$ has {\em trivial Sylow restriction}, if
the restriction of $M$ to a Sylow $p$-subgroup $S$ of $G$ 
has the form $M_{\downarrow S} \cong k \oplus \proj$.  Any such 
module is endotrivial and is 
the direct sum of an indecomposable trivial source module
and a projective module. 
\end{defi}

The following is a straightforward application
of Mackey formula (see Lemma~2.6 in~\cite{MT}).

\begin{prop}\label{prop:normal-p}
Let $S$ be a Sylow $p$-subgroup of $G$.
If $G$ has a 
nontrivial normal $p$-subgroup then every indecomposable $kG$-module
with trivial Sylow restriction has dimension one.  
\end{prop}

\begin{cor} \label{cor:normal1}
Suppose that $G = \GL(n,q)$ and $p$ divides $q-1$. 
Then every indecomposable $kG$-module with trivial Sylow restriction 
has dimension one. 
\end{cor}

\begin{proof}
With these hypotheses, the center of $G$ has order divisible 
by $p$ and so $G$ has a nontrivial normal $p$-subgroup. The result now
follows from Proposition~\ref{prop:normal-p}. 
\end{proof}

\begin{cor} \label{cor:normal2}
Suppose that $G = \SL(n,q)$ and $p$ divides 
both $q-1$ and $n$. If $M$ is an indecomposable 
endotrivial module with trivial Sylow restriction 
then $M$ is trivial.  
\end{cor}

\begin{proof}
In this case, the center of $G$ contains the scalar 
matrix $\zeta I_n$, where $I_n$ is the $n \times n$
identity matrix and $\zeta$ is a primitive $p^{th}$ 
root of unity. Thus by Proposition \ref{prop:normal-p},
$M$ has dimension one. Now observe that the only $kG$-module of 
dimension one is the trivial module since $\SL(n,q)$ is a
perfect group. 
\end{proof}

\begin{prop} \label{prop:normal2}
Suppose that $H \subseteq J$ are subgroups of $G$ 
such that $H$ contains a Sylow $p$-subgroup of $J$.
Assume that there is a $p$-subgroup $Q$ of $H$, with 
the property that every left coset of $H$ in $J$
is represented by an element in the normalizer of 
$Q$. That is, assume that $J = N_J(Q)\cdot H$. 
If $M$ is an endotrivial $kJ$-module such that
$M_{\downarrow H} \cong X \oplus \proj$ where $X$
has dimension one, then M has dimension one. In particular, 
if every indecomposable endotrivial $kH$-module with trivial Sylow
restriction has dimension one, then every indecomposable endotrivial
$kJ$-module with trivial Sylow restriction has dimension one. 
\end{prop}

\begin{proof}
Since $H$ contains a Sylow $p$-subgroup of $J$ it follows that $M$ is a 
direct summand of $X^{\uparrow J}$. By hypothesis, $Q$ acts 
trivially on $X$ and also 
trivially on 
\[
X^{\uparrow J} \ \cong \ kJ \otimes_{H} X 
\ \cong \ \sum_{x \in [J/H]} x \otimes X. 
\]
That is, $(X^{\uparrow J})_{\downarrow Q}$ is a direct sum
of copies of $k_Q$, while $M_{\downarrow Q} \cong k_Q \oplus \proj$.
Hence, $M$ has dimension one. 
\end{proof}

We end this section with a description of the structure 
of the group of endotrivial modules in the case that a 
Sylow $p$-subgroup of $G$ is cyclic. The main point is that the 
subgroup $N$ in the following theorem has the property $N \cap {}^g N$
is a $p^\prime$-group for any $g \notin N$. So the induction
and restriction functors define equivalences between the 
stable categories of $kG$- and $kN$-modules. 

\begin{thm}\label{thm:cyclic2}\cite[Theorem 3.6]{MT}\quad 
Suppose that $S$ is cyclic. Let $\wts$ be the unique subgroup
of $S$ of order $p$, and let $N =N_G(\wts)$.
\begin{itemize}
\item[(a)] $T(G)= \ \{[M^{\uparrow G}]~|~[M]\in T(N)\} \ \cong \ T(N)$.
\item[(b)] There is an exact sequence
$$
\xymatrix{ 0\ar[r] & X(N)\ar[r]& T(N) 
\ar[r]^{\res^N_S}& T(S) \ar[r] & 0 }
$$
where $X(N)$ denotes the group of isomorphism classes of the one
dimensional $kN$-modules. 
\end{itemize}
\end{thm}


\section{$TF(G)$ for general linear groups}\label{sec:GLback}

In this section, we begin by introducing some notation and recalling  
a few facts about the general linear groups and its subgroups. 
From this information we are able to deduce that the 
rank of the torsion free part of the group of endotrivial modules
is always one or zero. The zero case only occurs when the 
Sylow $p$-subgroup is cyclic. The basic premise of the proof is 
that every elementary abelian $p$-subgroup of $\GL(n,q)$ 
is conjugate to a subgroup of a specific standard subgroup. 
This has been observed before in other contexts. We sketch a proof for the 
sake of completeness. General reference material can be
found in the standard textbooks~\cite{AM,GLS,W}. 

We set some notation that is used throughout the paper. 
\begin{notation} \label{note:basic}
Let $n$ is a positive integer, and $q$ a power of a prime such that $\gcd(p,q)=1$ where $p$ is the 
characteristic of the field $k$. Let $e$ denote the least integer such that $p$ divides
$q^e-1$. Thus, $e$ is the smallest integer such that 
$p$ divides the order of $\GL(e,q)$. Let $r, f$ be 
integers such that $n = re+f$ for $0 \leq f < e$. 
\end{notation}

A Sylow $p$-subgroup $S$ of $G$ is a direct product of 
iterated wreath products.  From~\cite[Section~4.10]{GLS} a conjugate of $S$
is contained in the normalizer of an appropriate torus of $G$.
In particular, for $G = \GL(n,q)$, a Sylow subgroup of $G$ can 
be realized as the Sylow $p$-subgroup of a semi direct
product $(C_{p^t})^{\times r} \rtimes \CS_r$, where $\CS_r$ is the 
symmetric group on $r$-letters and $p^t$ is highest power
of $p$ dividing $q^e-1$. For $G = \SL(n,q)$, a Sylow $p$-subgroup
is the intersection of $G$ with a Sylow $p$-subgroup of $\GL(n,q)$. 
The exact structure is not important for this paper, except
that it is helpful to know the following consequence of these observations.

\begin{lemma} \label{lem:whenabel}
A Sylow $p$-subgroup of $\GL(n, q)$ or of $\SL(n,q)$ is 
abelian if and only if $n < pe$. 
\end{lemma}

We also need the following fact about the elementary abelian
$p$-subgroups of $G$. For notation, let $E_p$ be an elementary abelian $p$-subgroup 
of $G = \SL(n,q)$ or $G = \GL(n,q)$ constructed as follows. 
In the case that $p$ divides $q-1$ 
(so $e = 1$), we let $E_p$ be the intersection 
of $G$ with the subgroup of all diagonal matrices whose entries are  $p^{th}$ roots of unity in $\bF_q$. 
Otherwise, we let $L = L_1 \times L_2 
\times \dots \times L_r \times L_{r+1}$ be the Levi subgroup of $G$
consisting of $e \times e$ diagonal blocks, 
except for the last one which is
an $f \times f$ block if $f > 0$ and is 
empty otherwise. Let $E_p$ be a maximal elementary abelian 
$p$-subgroup of $L$. It is the subgroup of elements of order
dividing $p$ in a Sylow $p$-subgroup of $L$. It
has order $p^r$ and is generated by 
elements that are the identity in all but one of the blocks of $L$.
Note that for each $1\leq i\leq r$ any two subgroups of order $p$ in
$L_i$ are conjugate since the Sylow $p$-subgroups of $L_i$ are
cyclic. More extensive details are given in Section \ref{sec:GLinfo}.

\begin{lemma} \label{lem:GLeleab}
Suppose that $E \subseteq \SL(n,q)$ is an elementary abelian 
$p$-subgroup. Then $E$ is conjugate to a subgroup of $E_p$. 
\end{lemma}

\begin{proof}
Let $\bV$ denote the natural module for $G$. The restriction
$\bV_E$ of $\bV$ to a $\bfq E$-module is completely reducible 
by Maschke's Theorem. In the case that $p$ divides $q-1$ all
irreducible modules have dimension one. The change of basis 
matrix for the direct sum decomposition conjugates $E$ to a 
subgroup of $E_p$. In the case that $e>1$, the absolutely 
irreducible characters of $E$ are maps into $\bF_{q^e}^\times$ and
there is a copy of $\bF_{q^e}^\times$ in $\SL(e,q)$. Hence, all of 
the simple $\bfq E$-modules have dimension $e$ or $1$. The only one 
of dimension one is the trivial module. Again, the change of
basis matrix that creates the direct sum decomposition conjugates
$E$ into an elementary abelian $p$-subgroup of $L$. This 
can then be conjugated into $E_p$.
\end{proof}

This information settles the issue of the torsion free part
of the group of endotrivial modules. Note that this theorem 
does not depend upon the Sylow $p$-subgroup being abelian. 

\begin{thm} \label{thm:torfree}
Let $G$ be a group such that $\SL(n,q) \subseteq G \subseteq
\GL(n,q)$. Suppose that the Sylow $p$-subgroup of $G$ has 
rank at least $2$. Then $TF(G) \cong \bZ$ is generated by 
the class of $\Omega(k)$.
\end{thm}

\begin{proof}
The proof is a direct consequence of Theorem \ref{thm:rank} using 
Lemma \ref{lem:GLeleab}.
\end{proof}

We end this section with a few words about Young modules. The 
setting for general linear groups is described in ~\cite{ES}.
Let $G=\GL(n,q)$. A Young module is an indecomposable direct summand 
of an induced module $k_L^{\uparrow G}$ where $L$ is a Levi subgroup of $G$. 
The Young module corresponding to $L$ is indexed by
the same partition of $n$ that determines $L$. That is, if
$L=\GL(n_1, q)\times\dots\times\GL(n_t,q)$, for
$n=n_1+\dots+n_t$ with $n_1\geq \dots\geq n_t$, then 
for $\lambda=[n_1,\dots,n_t]$, the Young module $Y_\lambda$ 
is a distinguished direct summand of the induced module
$k_L^{\uparrow G}$. In particular $Y_{[n]} = k$ for the trivial
partition $[n]$ of $n$.

The key facts needed about Young modules are 
summarized in the following theorem.

\begin{thm} \label{thm:youngbasic} \cite{ES}
Suppose that $G = \GL(n,q)$ and $L$ is a Levi subgroup 
of $G$ corresponding to a partition $\lambda$ of $n$.
Then, 
\[
k_L^{\uparrow G} \quad \cong \quad Y_\lambda \oplus 
\bigoplus_{\mu > \lambda} (Y_\mu)^{\times a_\mu},
\]
for some multiplicities $a_{\mu}$. 
That is, $k_L^{\uparrow G}$ is a direct sum of Young modules
with exactly one of them being $Y_\lambda$. 
The other direct summands are Young modules that 
are indexed by partitions which are greater than $\lambda$ 
in the dominance ordering.
\end{thm}


\section{Abelian Sylow $p$-subgroups when $p$ divides $q-1$} 
\label{sec:abelian1}
In this section we assume that $G= \SL(n,q)$ and that the field
$k$ has characteristic $p$ dividing $q-1$. We assume further that the 
Sylow $p$-subgroup of $G$ is abelian and not cyclic. This requires that 
$2<n<p$. Our goal is to show that every indecomposable 
$kG$-module with trivial Sylow restriction has dimension one. 

Let $T$ be the torus of diagonal matrices with 
determinant one, and let $N$ be its normalizer in $G$. It is well
known that the Weyl group $N/T$ is isomorphic to the symmetric
group $\CS_n$. By~\cite[Theorem~4.10.2]{GLS}, our hypotheses also says 
that $T$ contains a Sylow $p$-subgroup $S$ of $G$. The first result is crucial to our arguments. 

\begin{prop} \label{prop:toral1}
Suppose that $M$ is a $kG$-module with trivial Sylow restriction. 
Then the restriction of $M$ to $T$ has the form,
$M_{\downarrow T} \cong k \oplus \proj$.
\end{prop}

\begin{proof}
Note that the Sylow subgroup $S$ of $T$ is normal in $N$. Consequently, 
$M_{\downarrow N} \cong X \oplus R$ where $X$ has dimension one and 
$R$ is a projective $kN$-module.  We observe that $M_{\downarrow T}
\cong X_{\downarrow T} \oplus R_{\downarrow T}$, and $R_{\downarrow T}$
is a projective $kT$-module. So the proof is complete if we show that
$X_{\downarrow T}$ is a trivial $kT$-module. 

The representation $N \longrightarrow \GL(X)$ that affords $X$ has 
the commutator subgroup $[N,N]$ of $N$ in its kernel, since $\GL(X)$ is 
commutative. Choose a generator $\zeta$ for the multiplicative group 
$\bF_q^\times$ of nonzero elements of $\bF_q$.
The group $T$ is generated by diagonal matrices 
$D_1, \dots, D_{n-1}$ where $D_i$ has 
diagonal entries $1, \dots, 1, \zeta,
\zeta^{-1}, 1, \dots, 1$, the entry $\zeta$ 
being in the $i^{th}$ position. 
Let $\sigma_i$ be the block matrix 
\[
\sigma_i \quad = \quad \begin{pmatrix} I_{i-1} & 0 & 0 \\
0 & U & 0 \\
0 & 0 &  I_{n-i-1} \end{pmatrix},
\]
where $I_j$ is the $j \times j$ identity matrix and 
\[ 
U = \begin{pmatrix} 0 & -1 \\ 1 & 0 \end{pmatrix}.
\]
If $i < n-1$, let $W = D_{i+1}$, and if $i = n-1$, 
let $W = D_{i-1}$. We
then get the equality $\sigma_iW\sigma_i^{-1}W^{-1} = D_i$. 
Therefore $T \subseteq [N,N]$, and $T$ acts trivially on $X$.
This completes the proof. 
\end{proof}

For later use we record the following lemma which is proved
above. 

\begin{lemma} \label{lem:Tincommutator}
Suppose that $G = \SL(n,q)$ where $p$ divides $q-1$ and $3\leq n < p$.
Then for $T$ and $N$ as above, we have that $T \subseteq [N,N]$.
\end{lemma}

Next, let $L =(\GL(n-1,q) \times \GL(1,q)) \cap G$ be
the Levi subgroup of all matrices of determinant one that can be put
into block form of an upper left $(n-1)\times (n-1)$  block and a lower
right $1 \times 1$ block.

\begin{prop} \label{prop:toral2}
Assume that $p$ divides $q-1$.
Let $M$ be a $kG$-module with trivial Sylow restriction. 
Then the restriction of $M$ to $L$ has the form,
$M_{\downarrow L} \cong k_L \oplus \proj$.
\end{prop}

\begin{proof}
Suppose that $\zeta \in \bF_q$ is a primitive $p^{th}$ root of unity.
Let $z \in G$ be the diagonal matrix with diagonal entries
$\zeta, \zeta, \dots, \zeta, \zeta^{1-n}$. 
Note that $z\notin Z(G)$ because our assumption $2<n<p$ implies that
$1-n\not\equiv1\pmod p$. However the important point
is that $z$ generates a subgroup of order $p$ that is 
central in  $L$.  In particular, $L$ has a nontrivial normal 
$p$-subgroup. Hence, $M_{\downarrow L} \cong X \oplus \proj$
where $X$ has dimension one by Proposition \ref{prop:normal-p}. 
By Proposition \ref{prop:toral1}, it suffices to show that the restriction of $X$ to $T$ is the trivial
module if and only if $X$ is the trivial $kL$-module. 

We observe that the homomorphism $\psi:\GL(n-1,q) \longrightarrow
\SL(n,q)$, given by sending an invertible matrix $A$ to the matrix
\[
A \quad \mapsto  \quad \begin{pmatrix} A & 0 \\ 0 & \Det(A)^{-1} 
\end{pmatrix},
\] 
induces an isomorphism $\GL(n-1,q) \cong L$. The one dimensional
representations of $\GL(n-1,q)$ over $k$ are well known and are
the maps  
\[
\varphi_i: \GL(n-1,q) \ \longrightarrow \ k^\times  \quad \text{ given by }
A \mapsto \Det(A)^i
\] 
for $i = 0, 1, \dots, q-2$.
It is easy to see that the only $\varphi_i$ which has $T$ in its kernel
is $\varphi_0$. This finishes the proof. 
\end{proof}

Before proceeding to our main result, we record a known 
fact involving the numbers of cosets for an arbitrary group $H$ that will be needed later.  
We give a sketch of the proof for the sake of completeness.  

\begin{lemma} \label{lem:count1}
Suppose that $U$ and $V$ are subgroups of a finite group $H$ such 
that $U \subseteq V$. Then the number of $U$-$V$ double cosets 
$UxV$ in $H$ that are left cosets of $V$ is the same as the 
number of double cosets $UxV$ in $H$ such that $U \subseteq xVx^{-1}$.
This number is at least as large as $\vert V \vert / \vert N_V(U) \vert.$
\end{lemma}
Another way to cast Lemma~\ref{lem:count1} is to say that the number of
left cosets $xV$ of $H/V$ that are stabilized by the action of $U$ (on
the left) is equal to the index of $N_V(U)$ in $V$.

\begin{proof}
Note that for $x$ in $H$, we have that $UxV$ is a left coset of $V$ in $H$
if and only if $UxV = xV$. This  is equivalent to saying that $x^{-1} Ux
\subseteq V$ or that $U \subseteq xVx^{-1}$. Now note that if 
$x \in N_H(U)$, then $UxV = xV$. Hence, the number of such double
cosets is at least equal to 
$$\frac{\vert N_H(U)\cdot V \vert}{\vert V \vert}=\frac{\vert
  V\vert}{\vert N_H(U)\cap V\vert}=\frac{\vert V \vert}{\vert N_V(U) \vert}$$
\end{proof}

Proposition~\ref{prop:toral2} shows
that if $p$ does not divide $n$ and if $M$ is an indecomposable
$kG$-module with trivial Sylow restriction then the restriction of  
$M$ to $L$ is the direct sum of a trivial 
module and a projective module. Because $L$ contains a Sylow $p$-subgroup
of $G$, it follows from standard arguments on relative projectivity
that $M$ is a direct summand of $k_L^{\uparrow G}$, the induction of 
the trivial $kL$-module to $G$. Now by the theory of 
Young modules for the general linear group (see
Theorem \ref{thm:youngbasic}), we have that 
\begin{equation} \label{eq:young}
k_\whl^{\uparrow \whg} \quad \cong \quad Y_{[n-1,1]} \ \oplus \ Y_{[n]}, 
\end{equation}
where $\whg = \GL(n,q)$ and $\whl = \GL(n-1,q) \times \GL(1,q)$
is the Levi subgroup in $\whg$. Indeed Theorem
\ref{thm:youngbasic} asserts that 
$k_\whl^{\uparrow \whg} \cong Y_{[n-1,1]} \oplus (Y_{[n]})^{\times m}$
for some $m\geq 1$, because $[n]$ is the only
partition of $n$ greater than $[n-1,1]$. Now $Y_{[n]}=k$ 
and so we get that $m=1$ by Frobenius reciprocity. 
Hence once we show that the restriction to $G$ of the 
indecomposable module $Y_{[n-1,1]}$ has no endotrivial direct summands we 
will be able to deduce the following result. 

\begin{thm} \label{thm:sl-pdiv}
Suppose that $G = \SL(n,q)$ with $2 < n < p$ and that $p$ divides 
$q-1$. Then $TT(G) = \{0\}.$
\end{thm}

\begin{proof}
Recall from Section~\ref{sec:GLback} that the assumption $2 < n < p$ says
that a Sylow $p$-subgroup $S$ of $G$ is abelian and not cyclic. 
More precisely, if $p^t$ is the highest power of $p$ dividing $q-1$, then
$S\cong (C_{p^t})^{n-1}$ is homocyclic. For convenience, let $\whg
= \GL(n, q)$ and $\whl\cong\GL(n-1,q)\times\GL(1,q)$
for the usual Levi subgroup and set $L=\whl\cap G$.  
Note that $L\cong \GL(n-1,q)$ has
index $q-1$ in $\whl$.
We choose $S\leq T$, so that $N_G(S)\leq N$. By
Proposition~\ref{prop:toral2}, the group $TT(G)$ is 
generated by the classes of the indecomposable endotrivial modules
that are direct summands of the induced module $k_L^{\uparrow G}$. 

We start by noticing that $k_L^{\uparrow G}\cong 
(k_\whl^{\uparrow \whg})_{\downarrow G}$. In order to see this,  apply the Mackey formula to
$$
(k_\whl^{\uparrow \whg})_{\downarrow G} \quad \cong \quad
\sum_{g \in [G\backslash\whg/\whl]} (k_{{}^g\whl \cap G})^{\uparrow G}
= k_L^{\uparrow G}.
$$
The last equality is a consequence of the facts that
$\whl G = \whg$ and there is only one $\whl$-$G$ 
double coset in $\whg$, represented by $g=1$.

From Equation~(\ref{eq:young}) 
$$k_\whl^{\uparrow \whg} \quad \cong \quad Y_{[n-1,1]}\oplus Y_{[n]}$$
so that $k_L^{\uparrow G}\cong (Y_{[n-1,1]})_{\downarrow G} \oplus k$.
We are left with the task of proving that 
$(Y_{[n-1,1]})_{\downarrow G}$ has no endotrivial direct
summands. 

Now let $U=\langle\sigma\rangle$ with 
$\sigma\in G$ the diagonal matrix with
entries $\zeta^{1-n},\zeta,\dots,\zeta$ where $\zeta$ is a primitive
$(q-1)$-st root of unity. 
Because $2<n<p$, we have that $\zeta\neq\zeta^{1-n}$
and the normalizer $N_L(U)$ of $U$ in $L$ is the
intersection of the Levi subgroup
$\GL(1,q)\times\GL(n-1,q)$ of $\whg$ with
$L$.  Hence, the intersection is isomorphic to 
$\GL(1,q)\times\GL(n-2,q)$, and
$$
N_G(U)=\big(\GL(1,q)\times\GL(n-1,q)\big)\cap G\cong
\GL(n-1,q).
$$
We claim that 
$(Y_{[n-1,1]})_{\downarrow U}$ is not an endotrivial
$kU$-module. The Mackey formula yields
$$
(k_L^{\uparrow G})_{\downarrow U}
\quad \cong \quad \sum_{g\in [U\backslash G/L]}
(k_{\ls gL\cap U})^{\uparrow U}.
$$
By Lemma~\ref{lem:count1} 
\[ 
\vert [U\backslash G/L] \vert \geq
\vert L:N_L(U) \vert >  2
\]
and so there are at least three trivial direct summands
in $(k_L^{\uparrow G})_{\downarrow U} \cong (Y_{[n-1,1]})_{\downarrow U}
\oplus k_{\downarrow U}$. This is because we can choose double coset
representatives for $[U\backslash G/L]$ in $N_G(U)$, 
which implies that $\ls
gL\cap U=U$ for each $g\in[U\backslash G/L]$. 
So the multiplicity of $k_U$ as a direct summand of
$(Y_{[n-1,1]})_{\downarrow U}$ is at least two, which proves that
$(Y_{[n-1,n]})_{\downarrow U}$ is not endotrivial. 

We claim that this is enough to prove the theorem. 
Suppose that $(Y_{[n-1,n]})_{\downarrow G}$
has an endotrivial direct summand.  
To show the claim note that $S$ is characteristic
in $T$ and hence normal in $N$. So any endotrivial $kN$-module 
having trivial Sylow restriction has dimension one. 

Since $N \cong T \rtimes \CS_n$ and $T \subseteq [N,N]$, there are exactly
two distinct one dimensional $kN$-modules, namely $k$ and 
the sign module.  Then $k_T^{\uparrow N}$ 
has exactly two one-dimensional direct summands. As a consequence, 
$k_T^{\uparrow G} = (k_T^{\uparrow N})^{\uparrow G}$ can have
at most two endotrivial direct summands (the Green correspondents
of the one dimensional modules). Now $k_L^{\uparrow G}$ is a 
direct summand of $k_T^{\uparrow G}$, so it can have at most
two endotrivial directs summand, one of which is $k$. 

By Clifford theory, $(Y_{[n-1,1]})_{\downarrow G}$ is 
either indecomposable or
the direct sum of conjugate modules. If the latter holds, then 
one of the summands is 
endotrivial and the conjugates are also. We know that 
$(Y_{[n-1,1]})_{\downarrow G}$ can have only one endotrivial
indecomposable direct summand and $(Y_{[n-1,1]})_{\downarrow G}$ must be
indecomposable and endotrivial. However, this is not possible as the
restriction to $U$ is not endotrivial.  
\end{proof}


\section{The kernel of restriction} \label{sec:kernel}
In this section we present a method introduced by Balmer \cite{Bal} (see also \cite{Bal2})
for computing the kernel of the restriction map $T(G) \to T(H)$
in the case that $H \subseteq G$ is a subgroup that contains a 
Sylow $p$-subgroup of $G$. We also develop some consequences of 
these basic ideas that will be used in later sections. The  main definition is the 
following.

\begin{defi} Suppose that $H$ is a subgroup of $G$ that contains 
$S$, a Sylow $p$-subgroup of $G$. A function $u: G \to k^{\times}$ 
is a {\it weak $H$-homomorphism} if it satisfies the three conditions:
\begin{itemize}
\item[(a)] if $h \in H$, then $u(h) = 1$,
\item[(b)] for $g \in G$, if $\vert H \cap \ls gH \vert$ has order prime
to $p$, then  $u(g) = 1$, and 
\item[(c)] if $a, b \in G$ and $p$ divides $\vert H \cap \ls aH \cap 
\ls{ab}H \vert$ then $u(ab) = u(a)u(b)$. 
\end{itemize}
\end{defi}

Let $A(G,H)$ be the set of all weak $H$-homomorphisms. 
It is easy to show that $u(g^{-1}) = (u(g))^{-1}$ for every 
$u \in A(G,H)$ and all $g \in G$. Furthermore, one can verify that the set $A(G,H)$ is a
group under composition of functions. If $G$ has a nontrivial normal
$p$-subgroup, then every weak $H$-homomorphism is a homomorphism whose
kernel contains $H$. Balmer proved an amazing theorem which relates  
$A(G,H)$ and the kernel of the restriction map from $T(G)$ to $T(H)$. 

\begin{thm} \label{thm:balmer} \cite{Bal}
Suppose that $H$ is a subgroup of $G$ of index prime to $p$.
Then $A(G,H)$ is isomorphic to the kernel of the restriction
map $T(G) \to T(H)$. 
\end{thm}

We use that theorem to establish two sets of criteria for the 
vanishing of the kernel of restriction on endotrivial modules. First
we need the following basic lemma. Part (a) is proven in \cite{Bal}. We 
include a proof for the convenience of the reader. 

\begin{lemma} \label{lem:kerbasics}
Let $H$ be a subgroup of $G$ having index prime to $p$. Suppose
that $u: G \to k^{\times}$ is a weak $H$-homomorphism. Then 
\begin{itemize}
\item[(a)] $u$ is constant on $H$-$H$ double cosets (meaning 
that if $HaH =HbH$ then $u(a) = u(b)$), 
\item[(b)] if $a, b \in N_G(Q)$ for some nontrivial $p$-subgroup $Q \subseteq H$,
then $u(ab) = u(a)u(b)$, and 
\item[( c)] if $g$ is in the commutator subgroup $[N,N]$ of $N = N_G(Q)$,
for some nontrivial $p$-subgroup $Q \subseteq H$, then $u(g) = 1$. 
\end{itemize}
\end{lemma}

\begin{proof}
Note that part (b) follows directly from the definition of a weak
$H$-homomor\-phism, and part (c) is a trivial consequence of (b). So 
only part (a) needs some proof. For this assume that $a, b \in G$
and that $HaH = HbH$, so that $b = h_1ah_2$ for some $h_1, h_2 \in H$. 
There are two cases to consider. 

First, if $\vert H \cap \ls aH \vert$ is prime to $p$ then so also 
is $\vert H \cap \ls{h_1ah_2}H\vert = \vert (H \cap \ls aH)^{h_1^{-1}}
\vert$. So the lemma holds in this case. 

Second, suppose that $\vert H \cap \ls aH \vert$ is divisible by 
$p$. Then clearly, $\vert H \cap \ls{h_2}H \cap \ls{ah_2}H \vert$
has order divisible by $p$. So $u(ah_2) = u(a)$. Also, there exists
a nontrivial $p$-subgroup $Q$ in $H \cap \ls aH$. Thus $\ls{h_1}Q$ 
is in $H \cap \ls{h_1}H \cap \ls{h_1a}H$. So $u(h_1a) = u(a)$, as
desired. 
\end{proof}

Our first application is almost an immediate consequence of the 
above lemma, in particular parts (a) and (c). 

\begin{prop} \label{prop:balmerapp1}
Let $H$ be a subgroup of $G$ having index relatively prime to $p$.
Suppose that every $H$-$H$ double coset $HxH$ in $G$ with $x\notin H$
has the property that either
\begin{itemize}
\item[(a)] $p$ does not divide $\vert H \cap \ls xH \vert$, or 
\item[(b)] $HxH = HaH$ for some $a$ in $[N_G(Q), N_G(Q)]$, the commutator
subgroup of $N_G(Q)$, for some nontrivial $p$-subgroup $Q$ of $H$. 
\end{itemize}
Then $A(G,H) = \{1\}$.
\end{prop}

The next application applies only to the case that $H= S$ 
is a Sylow $p$-subgroup of $G$.

\begin{prop} \label{prop:balmerapp2}
Let $S$ be a Sylow $p$-subgroup of $G$. 
Suppose that the following two conditions are satisfied.
\begin{itemize}
\item[(a)] For any element $u$ in $A(G,S)$, we have that $u(x) = 1$ whenever
$x \in N_G(Q)$ for some nontrivial subgroup $Q$ of $S$.
\item[(b)] If $Q_1$ and $Q_2$ are subgroups of $S$ that are conjugate in
$G$, then for some nontrivial subgroup $Q$ of $S$, there exists an
element $g \in N_G(Q)$ such that $\ls gQ_1 = Q_2$.
\end{itemize}
Then $A(G,S)= \{1\}$. 
\end{prop}

\begin{proof}
Let $x$ be any element of $G$ such that $S \cap \ls xS = Q \neq \{1\}$.
Our purpose is to show for any $u \in A(G,S)$, that $u(x) = 1$.
This would prove the proposition.  There are exactly two things that can
happen in this situation. The first is that $x \in N_G(Q)$, in which
case $u(x) = 1$ by condition (a) of the hypothesis. Hence, we are left
to assume the second possibility that $\ls xQ \neq Q$.  Let
$R = \ls{x^{-1}}Q$. Then $\ls xR = Q$.  By condition (b) of the
hypothesis, there exists an element $g \in N_G(\whq)$ for some nontrivial
subgroup $\whq$ of $S$ such that $\ls gQ = R$. Note, that by condition
(a), $u(g) = 1$. Then we have that $R =\ls{g}Q = \ls{x^{-1}}Q$, so
that $xg \in N_G(Q)$. Thus, $u(xg) = 1$ by condition (a).
Now note that $S \cap \ls xS \cap \ls{xg}S \supseteq Q$
which is not trivial. So from the definition we have
that $u(xg) = u(x)u(g)$ and $u(x) = 1$ as desired.
\end{proof}

In special cases, we can relax the conditions 
somewhat as in the following.

\begin{cor} \label{cor:balmerapp3}
Suppose that the Sylow $p$-subgroup $S$ of $G$ is abelian. Let 
$E$ be the
subgroup of $S$ of all elements with order $p$ (or $1$).
Assume also that the following condition is satisfied.
\begin{itemize}
\item For any element $u$ in $A(G,S)$, we have that $u(x) = 1$ whenever
$x \in N_G(Q)$ for some nontrivial subgroup $Q \subseteq E$.
\end{itemize}
Then $A(G,S)= \{1\}$.
\end{cor}

\begin{proof}
Observe that if $Q$ is any nontrivial subgroup of $S$, then the subgroup
$Q \cap E$, which contains all elements of order $p$ in $Q$,
has the property that $N_G(Q) \subseteq N_G(Q \cap E)$.  
Hence, the hypothesis implies condition (a) of Proposition
\ref{prop:balmerapp2}. Condition (b) is a consequence of the fusion 
theorem of Burnside (see \cite[Chapter 7, Theorem 1.1]{G}), which 
says that, because $S$ is abelian, $N_G(S)$ controls the fusion of
subgroups of $S$.
\end{proof}

We end this section with a proof of an inductive step in a scheme
to use the above results effectively. For notation, let 
$\Gamma_G$ denote the set of nontrivial
subgroups of $G$, and let $\Gamma_{H,p}$ denote the set of nontrivial
$p$-subgroups of $H$. 

\begin{prop} \label{prop:balmerapp4}
Let $H \subseteq G$ be a subgroup of index relatively prime to $p$.
Suppose that $\rho: \Gamma_{H,p} \to \Gamma_G$ is a function with 
the properties:
\begin{itemize}
\item[(a)] for every $Q \in \Gamma_{H,p}$, $\rho(Q) \subseteq N_G(Q)$, and 
\item[(b)] if $u \in A(G,H)$ and $Q \in \Gamma_{H,p}$, then $u(g) =1$ for 
all $g \in \rho(Q)$. 
\end{itemize}
Let $\widehat{\rho}: \Gamma_{H,p} \to \Gamma_G$  be defined by the 
rule 
\[
\widehat{\rho}(Q) \quad = \quad \langle N_G(Q) \cap \rho(Q^\prime)
\ \vert \ Q^\prime \in \Gamma_{H,p} \rangle,
\]
the subgroup generated by $\rho(Q)$ and all of the intersections
$N_G(Q) \cap \rho(Q^\prime)$ for $Q^\prime$ in $\Gamma_{H,p}$.
Then $\widehat{\rho}$ also satisfies the above properties. Namely,
\begin{itemize}
\item[(i)] for every $Q \in \Gamma_{H,p}$, we have that 
$\widehat{\rho}(Q) \subseteq N_G(Q)$, and 
\item[(ii)] if $u \in A(G,H)$ and $Q \in \Gamma_{H,p}$, 
then $u(g) =1$ for all $g \in \widehat{\rho}(Q)$. 
\end{itemize}
\end{prop}

\begin{proof}
That $\widehat{\rho}(Q) \subseteq N_G(Q)$ for all $Q \in \Gamma_{H,p}$
is obvious from the construction. To establish the second property, 
assume that $g \in \widehat{\rho}(Q)$ for some $Q$. Then $g = g_1 g_2 
\dots g_n$ for some $n$ and for $g_i \in N_G(Q) \cap \rho(Q_i)$ 
for some $Q_i \in \Gamma_{H,p}$. Thus, $u(g_i) = 1$ for all  $i =
1, \dots, n$. But because all of these elements are in $N_G(Q)$, 
we have that $u(g) = u(g_1) \cdots u(g_n) = 1$.
\end{proof}

\begin{rem} The function 
$\rho: \Gamma_{H,p} \to \Gamma_G$
given by $\rho(Q) = [N_G(Q),N_G(Q)]$ satisfies the 
hypothesis of the proposition by Lemma 
\ref{lem:kerbasics}(c).
\end{rem}


\section{Some information about $\GL(e,q)$} \label{sec:GLinfo}

The objective of this section is to set some notation for 
the remainder of the paper and to establish some elementary facts about
the group $G = \GL(e,q)$. We include sketches of some proofs for 
the sake of completeness. 

Throughout the section, let $\bF = \bfq$ and $\bK = \bF_{q^e}$, and assume Notation \ref{note:basic}. 
In addition, assume throughout the section that $e>1$. Notice that this 
assumption requires that $p >2$. Thus, $p$ divides $q^{e-1} + \dots + q + 1$.

Let $\bV$ denote the natural module for $G$. We can identify
$\bV$ with $\bK$, and we have an $\bF$-linear action of $\bK^\times$
on itself. Consequently, we have a $\bF$-linear embedding of $\bK^\times$
into $G$. There is an action of the Galois group of $\bK$
over $\bF$ by $\bF$-linear transformation. Consequently, the 
Galois group can also be embedded in $G$.

\begin{notation} \label{note:gens}
Let $w$ be a generator for $\bK^{\times}$, regarded as a subgroup of $G$.
Let $u \in \langle w \rangle$ be an element of order $p$. Let 
$g \in G$ denote a generator for the Galois group of $\bK$ 
over $\bF$, so that $gwg^{-1} = w^q$. Let $v$ be an element of 
$C_G(S)$ of determinant -1. (The existence of such an element 
$v$ is established in Lemma \ref{lem:glebasics}(c).)
\end{notation}

We assume that $u$ is the companion matrix of its minimal
polynomial over $\bfq$. The Sylow subgroup $S$ of $G$ is cyclic 
since the power of $p$ dividing $\vert G \vert$ is the same 
as that dividing $\vert \bK^\times\vert$ which is cyclic. Hence,
we can consider $S$ to be a subgroup of $\langle w \rangle =
\bK^\times \subseteq G$, and $u \in S$. Note, however, that 
$S$ need not be generated by $u$. For an example, let 
$G=\GL(5,3)$ with $p = 11$. In this case, $11^2$ divides $3^5-1$,
but $11$ does not divide $3^t-1$ for $t< 5$.

\begin{lemma} \label{lem:glebasics}
With the above notation, and assuming that $e > 1$, we have the following.
\begin{itemize}
\item[(a)] The centralizer of $u$ in $G$ is $\langle w \rangle$.
\item[(b)] The normalizer of $S$ in $G$ is generated by $w$ and $g$.
\item[(c)] The restriction of the determinant function $\Det: G
\to \bfq^{\times}$ to $\bK^\times \subseteq G$ is surjective.
Specifically, $\Det(w)= w^s$ where $s = 1 + q + \dots +q^{e-1}$ and
$w^s$ is a generator for the cyclic group $\bfq^\times$.
\item[(d)]  The commutator subgroup of $N_G(S)$ is generated by $w^{q-1}$.
It has index $e(q-1)$ in $N_G(S)$ and is equal to
the set of all elements of
$C_G(S)$ that have determinant 1.
\item[(e)] The determinant of $u$ is 1, and the determinant of $g$ is
$(-1)^{e-1}$.
\item[(f)] Suppose that $a$ is an $e \times e$ matrix over $\bfq$ such
that $au^\ell = u^ma$ for some $\ell, m$ such that $\ell$ is not divisible
by $p$. Then either $a$ is invertible or $a = 0$.
\end{itemize}
\end{lemma}

\begin{proof}
We note that the natural module $\bV$ is an irreducible module
for the group of order $p$ generated by $u$. This is because
an element of order $p$ acts on no smaller $\bfq$-vector
space by the minimality of $e$. So by Schur's Lemma, the
commuting ring of the action is a field, and we know that this
field contains $\bK$. On the other hand the algebra of
$e\times e$ matrices over $\bfq$ contains no extensions of
degree greater than $e$. This proves (a).

For (b), suppose that $x \in G$ normalizes $S$.
Then conjugation by $x$ is an
automorphism on the subalgebra of $e\times e$ matrices
generated by $w$, which is isomorphic to $\bK$.  This requires
that $x$ act on $w$ as some power of $g$, or that $x$ be equal
to some power of $g$ times an element of the centralizer of
$w$ which is $\langle w \rangle$. 

The minimal polynomial of $w$ over $\bfq$ is the product
$f(X) \ =  \prod_{i=0}^{e-1} (X - w^{q^i}).$
That is, the roots of this polynomial are $w$ and its Galois conjugates.
The constant term of $f(X)$, is $(-1)^ew^s$. The matrix of
$w$ over $\bfq$ is conjugate to the companion matrix of its minimal
polynomial whose determinant is $(-1)^e$ times the constant term of
the polynomial. Hence, we have that $\Det(w) = w^s$, as asserted.
The order of $w$ is $\vert \bK^\times \vert = q^e-1 =
(q-1)s$. Consequently, $w^s$ has order $q-1$ and $w^s$ is a generator
for $\bfq^\times$. This proves (c).

If $x$  is in $\langle w \rangle = C_G(S)$, then
$gxg^{-1}x^{-1} = x^qx^{-1} = x^{q-1}$. Thus, $[N_G(S), N_G(S)]
= \langle w^{q-1} \rangle$ as asserted. The second statement in (d)
follows from (c).

The first statement of (e) is a consequence of the fact that $p$ does
not divide $q-1$.  For the second, we note that there is a basis
of $\bK$ as an $\bfq$-vector space that is permuted by the
Galois automorphism $g$.  Consequently, viewing the natural
module $\bV$ as $\bK$, with respect to this basis, the element
$g$ acts by a permutation matrix, representing an $e$-cycle.

For (f) we observe that $au^\ell = u^ma$ implies that $u^\ell$ stabilizes
the nullspace of $a$. However,  we know that the group
$\langle u^\ell \rangle = \langle u \rangle$
has no nontrivial $\bfq$ representations of degree less than $e$.
So either the nullspace of $a$ is zero or it is the entire
natural module.
\end{proof}


\section{Endotrivial modules for $\SL(2e, q)$}\label{sec:case-2e}
In this section we prove that $TT(G) = \{0\}$ for $G = \SL(2e, q)$,
where $e$ is the least integer such that $p$ divides $q^e-1$ and $p>2$. 
Our strategy is to show that the group of weak $S$-homomorphisms
$A(G,S)$ is trivial. This implies the kernel of the restriction map is trivial and 
the restriction map from $T(G) \to T(S)$ is a monomorphism where $S$ is a Sylow $p$-subgroup 
of $G$. In the proof we verify the condition in Corollary 
\ref{cor:balmerapp3}, by an inductive process described in 
Proposition \ref{prop:balmerapp4}. This requires a detailed 
analysis of the normalizers of the $p$-subgroups in 
$S$. 

For notational convenience, we express elements of $G$ as $2 \times 2$ block
matrices where the blocks have size $e\times e$, and keep the 
notation of the previous section. In particular, the elements
$w, u, g$ and $v$ are as in Notation \ref{note:gens}. Let $H$ denote the group $\GL(e,q)$. 
The Sylow $p$-subgroup $S$ of $G$ has order $p^{2t}$
where $p^t$ is the highest power of $p$ dividing $q^e-1$.
The group $S$ is the direct product of two cyclic groups of order $p^t$, and 
has a maximal elementary abelian subgroup $E$ of order $p^2$.  We can assume $E$ is generated by
\[
A \ \ = \begin{bmatrix} u & 0 \\ 0 &1 \end{bmatrix}
\quad \text{and} \quad
B \ \ = \begin{bmatrix} 1 & 0 \\ 0 &u \end{bmatrix}.
\]

The subgroups of order $p$ in $E$ are naturally divided 
into three types. We give representatives as 
\[
(1.) \ \ Q_1 = \langle A \rangle, \qquad
(2.) \ \ Q_2  = \langle AB \rangle, \quad \text{and} \quad
(3.) \ \ Q_3  = \langle AB^m \rangle, \qquad
\]
where in the third case, $m$ is an integer such that
$u^m$ is not conjugate to $u$ in $\GL(e, q)$. Note that in some
situations such as when $G = \SL(4, 5)$ and $p=3$, there is
no such $m$ and the third case does not exist. Note further
that $Q_1$ is conjugate to $\whq_1=\langle B \rangle$ and that
$Q_2$ is conjugate to $\langle AB^{q^\ell}\rangle$ 
for any $\ell$. Hence, we have
the following.

\begin{lemma} \label{lem:conjclass}
Every subgroup of order $p$ in $G$ is conjugate to $Q_1$
or to $Q_2$ or to a subgroup of $E$ of type $Q_3$. 
\end{lemma}

There is one particular subgroup that plays an important 
role in this development. Namely, 
\[
U = \left\{ \begin{bmatrix} a & 0 \\ 0 & b \end{bmatrix}
\vert \ a, b \in C_H(u) = \langle w \rangle, \ \Det(ab) = 1 \right\}.
\]
We note that $U$ is in the centralizer of every subgroup 
of $S$.

First we consider the centralizers of these subgroups. Recall that
$H = \GL(e,q)$.

\begin{prop}
The centralizers of the subgroups of order $p$ are given
as follows.
\begin{align*}
C_G(Q_1) \quad &= \quad \left\{
\begin{bmatrix} a & 0 \\ 0 & b \end{bmatrix} \ \vert \
a \in C_H(u), \ b \in H \ \text{ and } \ \Det(a)\Det(b) = 1
\right \}.\\
C_G(Q_2) \quad &= \quad \left\{ X =
\begin{bmatrix} a & b \\ c & d \end{bmatrix} \ \vert \
a, b, c, d \in C_H(u) \cup \{0\} \ \text{ and } \ \Det(X) = 1
\right \}.\\
C_G(Q_3) \quad &= \quad U.
\end{align*}
In the second case, $C_G(Q_2)$ is a subgroup of $\GL(2, \bK)$
that contains $\SL(2, \bK)$ as a subgroup of index $q-1$.
\end{prop}

\begin{proof}
The proofs are straightforward exercises.  
One note is that in the statement for $C_G(Q_2)$ by Lemma
\ref{lem:glebasics}(f), each of $a, b, c, d$ is either invertible or
is zero. The embedding of $\GL(2,\bK)$, is defined by the
embedding of $\bK$ in $H$ as observed in the previous
section. The subgroup $\SL(2, \bK)$ is the commutator
subgroup of $\GL(2,\bK)$ and any commutator (as a
$2e \times 2e$ matrix) must have
determinant one. Consequently, $\SL(2,\bK)$ is contained
in $C_G(Q_2)$. The fact about the index is a consequence
of Lemma \ref{lem:glebasics}(c).
\end{proof}

\begin{cor} \label{cor:centS}
The centralizer of $S$ and of $E$ is $C_G(S) = C_G(E) = U$.
\end{cor}

\begin{proof}
The centralizer of $E$ is the intersection of the
centralizers of $Q_1$ and of $\whq_1 =
\langle B \rangle$.
\end{proof}

Now recall from the last section that $v \in C_H(u) =
\langle w \rangle $ is an
element with $\Det(v) = -1$, and $g \in N_H(\langle u \rangle)$
is an element that  acts on $\bK^\times = \langle w \rangle$,
as the Galois automorphism, $gug^{-1} = u^q$.

\begin{prop} \label{prop:e2normalizer}
The normalizers of the subgroups of order $p$ in $S$ are
given as follows.
\begin{itemize}
\item[(a)] $N_G(Q_1)$ is generated by $C_G(Q_1)$ and by the element
$\begin{bmatrix} g & 0 \\ 0 & v^{e-1} \end{bmatrix}$.
\item[(b)] $N_G(Q_2)$ is generated by $C_G(Q_2)$ and by the element
$\begin{bmatrix} g & 0 \\ 0 & g \end{bmatrix}.$
\item[(c)]
$N_G(Q_3)$ is generated by $C_G(Q_3)$ and by the element
$\begin{bmatrix} g & 0 \\ 0 & g \end{bmatrix},$
and possibly also an element of the form
$\begin{bmatrix} 0 & g^i \\ g^jv^{e(i+j+1)} & 0 \end{bmatrix}$.
\end{itemize}
Also, $N_G(S)$ is generated by $C_G(S)= U$ and
the elements
\[
\begin{bmatrix} g & 0 \\ 0 & v^{e-1} \end{bmatrix}
\qquad \text{and} \qquad
\begin{bmatrix} 0 & 1 \\  -1 & 0 \end{bmatrix}.
\]
\end{prop}

\begin{proof}
The normalizer $N_G(Q_1)$ also normalizes the commutator subgroup of
the centralizer $C_G(Q_1)$ which has the form
\[
[C_G(Q_1), C_G(Q_1)]  \quad = \quad 
\left\{ \begin{bmatrix} 1 & 0 \\ 0& b
\end{bmatrix} \ \vert \ b \in \SL(e,q) \right \}.
\]
The normalizer of  $[C_G(Q_1), C_G(Q_1)]$ is the Levi subgroup of
two $e \times e$ blocks. Then the normalizer of $Q_1$ is the
subgroup consisting of all elements of the form
\[
\begin{bmatrix} a & 0 \\ 0 & b \end{bmatrix}
\]
such that $a \in N_H(\langle u \rangle)$, $b \in H$ and $\Det(a)\Det(b)
= 1$ (recall that $H = \GL(e,q)$). This proves (a).

Any element in $N_G(Q_2)$ has the form
\[
\left[ \begin{matrix} a& b \\ c & d \end{matrix}\right]
\]
and satisfies
\[
\left[ \begin{matrix}a& b \\ c & d \end{matrix}\right]
\begin{bmatrix} u& 0 \\ 0 & u \end{bmatrix} \ = \
\begin{bmatrix} u^\tau& 0 \\ 0 & u^\tau \end{bmatrix}
\begin{bmatrix} a& b \\ c & d \end{bmatrix}
\]
for some value of $\tau$. Thus, $xu= u^\tau x$
for $x = a, b, c \ \text{ or } \ d$, and there exists some $\nu$ such that
\[
\begin{bmatrix} a& b \\ c & d \end{bmatrix}
= \begin{bmatrix} g^\nu& 0 \\ 0 & g^\nu \end{bmatrix} X
\]
where $X \in C_G(Q_2)$. So, we have shown (b).

Similarly, in case (c), for some $\tau$, we have
\[
\left[ \begin{matrix}a& b \\ c & d \end{matrix}\right]
\begin{bmatrix} u& 0 \\ 0 & u^r \end{bmatrix} \ = \
\begin{bmatrix} u^\tau & 0 \\ 0 & u^{r\tau} \end{bmatrix}
\begin{bmatrix} a& b \\ c & d \end{bmatrix}.
\]
If $\tau \equiv q^\nu$ modulo $p$, then $a = a^\prime g^\nu$, 
$d = d^\prime g^\nu$ for some
$a^\prime, d^\prime \in C_H(u)$. In this case it would
be necessary that $b = c = 0$. Otherwise, if $\tau$ is not
a power of $q$ modulo $p$, then we would have that
$a = d = 0$,  $cuc^{-1} = u^{r\tau}$, $bu^rb^{-1} = u^\tau$.
In this case, $r\tau$ would be a power of $q$
modulo $p$, and we would be in the other possibility.

The proof for the normalizer of $S$ is similar. That is,
it is clear from the structure of the centralizers that
the only subgroup of $S$ that is conjugate to $Q_1$
is $\whq_1 = \langle B \rangle$. So anything in the
normalizer of $S$ must either normalize both $Q_1$
and $\whq_1$ or interchange them.
\end{proof}

\begin{rem}
We should emphasize that the exceptional case (c) in the
proposition actually occurs. For example, if $G = \SL(10,2)$ and $p=31$,
then $gug^{-1} =u^2$ and we calculate
\[
\left[ \begin{matrix} 0& g^2 \\ 1 & 0 \end{matrix}\right]
\begin{bmatrix} u& 0 \\ 0 & u^{15} \end{bmatrix}
\begin{bmatrix} 0& 1 \\ g^{-2} & 0 \end{bmatrix} \ = \
\begin{bmatrix} u^{-2}& 0 \\ 0 & u \end{bmatrix} \ = \
\begin{bmatrix} u& 0 \\ 0 & u^{15} \end{bmatrix}^{-2}.
\]
\end{rem}

At this point we define functions, $\rho_0, \rho_1, 
\rho_2: \Gamma_{S,p} \to \Gamma_G$ as in Proposition
\ref{prop:balmerapp4}. We start with $\rho_0(Q) 
= [N_G(Q), N_G(Q)]$ for any nontrivial subgroup $Q$ of $S$. 
Then we set $\rho_1 = \widehat{\rho}_0$ as constructed
in that proposition. So $\rho_1(Q)$ is the subgroup of 
$N_G(Q)$ generated by all intersections 
$N_G(Q) \cap \rho_0(Q^\prime)$ for $Q^\prime \in \Gamma_{S,p}$.
Then inductively, let $\rho_2 = \widehat{\rho}_1$ 
constructed by the same process. 

\begin{prop} \label{prop:rhozero}
The groups $\rho_0(Q)$, for $Q$ in $E$ are given as follows.
\begin{itemize}
\item[(a)] $\rho_0(Q_1)$ has the form
\[
\rho_0(Q_1) \ = \ \left\{
\begin{bmatrix} a & 0 \\ 0 & b \end{bmatrix}
\vert \ a \in \langle w^{q-1} \rangle \ \text{ and } \ b \in \SL(e,q)
\right\}. 
\]
\item[(b)] The subgroup $\rho_0(Q_2)$ is generated by $\SL(2, \bK)$
and elements of the form $X^{[q]}X^{-1}$, for $X \in C_G(Q_2)$ 
where
\[
\begin{bmatrix} a & b \\ c & d \end{bmatrix}^{[q]} \ = \
\begin{bmatrix} a^q & b^q \\ c^q & d^q \end{bmatrix}.
\]
\item[(c)] The subgroup $\rho_0(Q_3)$ is a subgroup 
of $U$.
\item[(d)] $\rho_0(S)$ is generated by a subgroup
of $U$ and the element $\begin{bmatrix} g & 0 \\ 0 & g^{-1}
\end{bmatrix}$.
\end{itemize}
\end{prop}

\begin{proof}
This is a direct calculation using
Proposition~\ref{prop:e2normalizer}. 
\end{proof}

\begin{prop}
For every nontrivial subgroup $Q$ of $E$ we have that
$U \subseteq \rho_1(Q)$.
\end{prop}

\begin{proof} Let $Q$ be any nontrivial subgroup of
$E$ and let $X = \left[\begin{smallmatrix} a & 0 \\ 0 & b
\end{smallmatrix} \right]$ be an element of $U$. Then, regarding
$a$ and $b$ as elements of $\bK$, we must have that
$ab = c^{q-1}$ for some $c \in \langle w \rangle$, 
because the determinant
of $X$ is one. That is, we can write
\[
\begin{bmatrix} a & 0 \\ 0 & b \end{bmatrix} =
\begin{bmatrix} a & 0 \\ 0 & a^{-1} \end{bmatrix}
\begin{bmatrix} 1 & 0 \\ 0 & ab \end{bmatrix}
\]
and the latter matrix must have determinant one. But now
\[
\begin{bmatrix} a & 0 \\ 0 & a^{-1} \end{bmatrix} \
\in \ \rho_0(Q_2) \quad \text{and} \quad
\begin{bmatrix} 1 & 0 \\ 0 & ab \end{bmatrix}
\ \in \rho_0(Q_1).
\]
Because both of these elements are in $U \subseteq N_G(Q)$
we have that $X$ is in $\rho_1(Q)$.
\end{proof}

Next we note some information about a few particular
elements.

\begin{lemma}
We have the following
\begin{itemize}
\item[(a)]
\[
A_1 \ = \ \begin{bmatrix} 1 & 0 \\ 0 & gv^{e-1} \end{bmatrix}
\ \in \ \rho_0(Q_1),  \quad \text{and} \quad
A_2 \ = \ \begin{bmatrix} gv^{e-1} & 0 \\ 0 & 1 \end{bmatrix}
\ \in \rho_0(\whq_1);
\]
\item[(b)]
\[
B_1 \ = \ \begin{bmatrix} g & 0 \\ 0 & g \end{bmatrix}
\ \in \rho_2(Q_i)  \quad
\text{for} \quad i = 2,3, \quad \text{and} \quad
B_2 \ = \  \begin{bmatrix} 0 & 1 \\ -1 & 0 \end{bmatrix}
\ \in \rho_0(Q_2);
\]
\item[(c)]
\[
C_1 \ = \ \begin{bmatrix} w & 0 \\ 0 & w^{-1} \end{bmatrix}
\ \in \ \rho_0(Q_2), \quad \text{and} \quad
C_2 \ = \ \begin{bmatrix} w^{q-1} & 0 \\ 0 & 1 \end{bmatrix}
\ \in \rho_2(Q_2).
\]
\end{itemize}
\end{lemma}

\begin{proof}
The first statement is obvious, because $gv^{e-1} \in \SL(e,q)$.
Thus, $A_2$ is an element of $\rho_1(Q_1)$, and the
product $A_1A_2$ is in $\rho_2(Q_2)$ and $\rho_2(Q_3)$. Because
both $A_1$ and $A_2$ are in $N_G(S)$ we have that $B_1$ is in
$\rho_1(S)$ and hence also in $\rho_2(Q_2)$.  This proves the
first part of the second statement. The second part follows from the
fact that $B_2$ has determinant equal to one.  The first part of (c) is obvious, while the
second part follows directly from the
fact that $C_2$ is the commutator of
$A_2$ and $C_1$, both of which are in $\rho_1(\whq_1)$, where 
$\whq_1 = \langle B \rangle$.
Hence the commutator is in
$\rho_1(\whq_1) \cap N_G(Q_2) \subseteq \rho_2(Q_2)$.
\end{proof}

\begin{prop} \label{prop:rho2}
For any nontrivial subgroup $Q$ of $E$, we have that
$\rho_2(Q) = N_G(Q)$.
\end{prop}

\begin{proof}
Recall that $N_G(Q_1)$ is generated by $C_G(Q_1) = U$
and $A_2$. So $\rho_1(Q_1) = N_G(Q_1)$. Likewise,
$\rho_1(S) = N_G(S)$ because the normalizer of $S$ is generated
by $A_2$ and $B_2$. The normalizer of $Q_3$ is generated
by $C_G(Q_3) = U$ and $B_1$. So $\rho_2(Q_3) = N_G(Q_3)$.

For the normalizer of $Q_2$, we note that the element $C_2$
when regarded as an element of $\GL(2,\bK)$ has determinant
$w^{q-1}$. As $w$ is a generator of $\bK^\times$, we get that
$C_G(Q_2)$ is generated by $C_2$ and $\SL(2, \bK)$. Thus,
$N_G(Q_2)$, which is generated by $C_G(Q_2)$ and $B_1$,
is contained in $\rho_2(Q_2)$. This proves the last part.
\end{proof}

From all this we can derive the following.

\begin{thm} \label{thm:2e}
Suppose that $G = \SL(2e,q)$ for $e >1$, the least integer
such that $p$ divides $q^{e}-1$. Then $TT(G) = \{0\},$
and $T(G) = \bZ$ is generated by the class of $\Omega(k)$. 
\end{thm}

\begin{proof}
By Theorem \ref{thm:torfree} we need only show that 
$TT(G) = \{0\}$, which is equivalent to showing that
$A(G,S) = \{1\}$ by Theorem \ref{thm:balmer} and using the fact
that $TT(S) = \{0\}.$  The fact that $A(G,S) = \{1\}$ is 
proved in Proposition \ref{prop:rho2} and Corollary \ref{cor:balmerapp3}.  
\end{proof}


\section{Endotrivial modules for $\SL(re, q)$, $2 < r <p$} \label{sec:re}
We consider the endotrivial modules for the 
group $G = \SL(n, q)$ where $n = re$ for $3 \leq r < p$. Here,
as before, $e$ is the least integer such that $p$ divides 
$q^e-1$. We assume that $e > 1$. Let $\whg = \GL(n, q)$. 
Our purpose in this 
section is to show that $TT(G)$ has order one, that is, the 
only indecomposable endotrivial module that has trivial Sylow 
restriction is the trivial module.

To begin, we let $\whl = \whl(e, e, \dots, e) 
\subseteq \whg$ be the Levi subgroup 
of invertible $e \times e$ diagonal block matrices. 
Let $L = \whl \cap G$, the corresponding Levi subgroup
for $G$. We start with a technical lemma that will be 
essential to our argument. 

\begin{lemma} \label{lem:re1}
Let $N = N_G(L)$. With the above hypothesis, we have that 
$L \subseteq [N,N]$ and $N/[N,N]$ has order $2$. 
\end{lemma}

\begin{proof}
The structure of $\whn = N_\whg(\whl)$ is well known. It
has the form 
\[ 
\whn \ \cong \  \whl \rtimes \CS_r,
\]
where $\CS_r$ is the symmetric group on $r$ letters. 
An element of $\whl \rtimes \CS_r$ is a pair consisting
of a block diagonal matrix with blocks $A_1, \dots, A_r \in \GL(e, q)$
and an element $\sigma$ of $\CS_r$.
So we have a map
\[ 
\varphi: \ \whn  \ \longrightarrow  (\bF_q^{\times})^r \rtimes \CS_r
\]
where $\varphi((A_1, \dots, A_r, \sigma)) = 
(\Det(A_1), \dots, \Det(A_r), \sigma)$.
One can verify 
that $\varphi$ is a homomorphism. Note that the kernel of $\varphi$
is $\SL(e,q)^{\times r}$ which is a perfect group and hence is contained
in the commutator subgroup $[N,N]$. 

Now note that $(\bF_q^{\times})^r \rtimes \CS_r$ is naturally 
isomorphic to $N_{GL(r,q)}(T)$, the normalizer in $\GL(r,q)$ of the
torus $T$ of invertible diagonal matrices. 
From Lemma \ref{lem:Tincommutator}, we know that 
$T \subseteq [N_{GL(r,q)}(T), N_{GL(r,q)}(T)]$.
It is an easy check that
$\varphi(\whn) = N_{GL(r,q)}(T) \cap \SL(r,q)$. 
Hence, $N/[N,N] \cong \CS_r/[\CS_r,\CS_r]$ has order 2. 
\end{proof}

\begin{lemma} \label{lem:re2}
Every indecomposable $kL$-module with trivial Sylow restriction has
dimension one.  
\end{lemma}

\begin{proof}
Let $U = \GL(e,q)$, and let $S_1$ be a Sylow $p$-subgroup of $U$. 
Let $\whh = N_U(S_1) \times U^{\times (r-1)}
\subseteq \whl$ be the subgroup of diagonal $e \times e$ 
blocks in $\GL(n,q)$, the first of which is in $N_U(S_1)$
and the last $r-1$ being in $U$. Let $H = \whh \cap G$. 
Note that $H$ has a nontrivial normal $p$-subgroup. Hence, every 
$kH$-module with trivial Sylow restriction 
has dimension one. Note that $\whl \cong U^{\times r}$ and
we see that every left coset of $H$ in $L$ is represented
by an element having the block form $X \times \Id_U^{\times (r-1)}$, 
for some $X \in U$ of determinant one, where $\Id_U$
is the identity element for $U$. Every one of these elements 
centralizes a nontrivial $p$-subgroup of $H$, namely a subgroup 
of the form $S_{1}\times \Id_U^{\times (r-1)} $. 
The lemma now follows from Proposition \ref{prop:normal2}.
\end{proof}

The following two lemmas deal with situations when $M$ has trivial Sylow restriction. 

\begin{lemma} \label{lem:re3} Let $M$ is a $kG$-module
with trivial Sylow restriction (i.e., $M_{\downarrow S} \cong k \oplus \proj$.) 
The restriction of $M$ to $L$ has the form 
\[
M_{\downarrow L} \ \cong \ k \oplus \proj.
\]
\end{lemma} 

\begin{proof}
Let $N = N_G(L)$. Write $M_{\downarrow N} \ = \ V \oplus \proj$ 
where $V$ is an indecomposable $kN$-module
with trivial Sylow restriction. We must have
that $V_{\downarrow L} \cong X \oplus \proj$ where $X$ has dimension
one, by Lemma \ref{lem:re2}. 
Because $L$ is normal in $N$, and $V$ is a direct summand of 
$X^{\uparrow N}$, it follows from Clifford theory and the fact that
$V$ is endotrivial that $V$ has 
dimension one. Thus, $X \cong k$,
since $L$ is in the commutator subgroup of $N$. 
\end{proof}

\begin{lemma} \label{lem:re4} Let $M$ be a $kG$-module
with trivial Sylow restriction, and let $\CL = L((r-1)e, e)$ 
be the Levi subgroup that is the intersection
of $G$ with $\GL((r-1)e,q) \times \GL(e,q) \subseteq \GL(n,q)$. Then 
$M_{\downarrow \CL} \cong k \oplus \proj$. Hence, $M$ 
is a direct summand of $k_{\CL}^{\uparrow G}$.
\end{lemma}
 
\begin{proof}
This is another argument using the Mackey formula. Let 
$M_{\downarrow \CL} = V \oplus \proj$, where $V$ is indecomposable. 
Then by the previous lemma, $V$ is a direct summand of 
$k_L^{\uparrow \CL}$. However, all of the $L$-$L$ 
double cosets in $\CL$ can be chosen to be in the 
factor $\GL((r-1)e,q) \cap G$ which centralizes a nontrivial
$p$-subgroup in the factor $\GL(e,q) \cap G$. Consequently, 
$(k_L^{\uparrow \CL})_{\downarrow L}$
can have no projective summands by the Mackey formula. 
Thus, $V \cong k_{\CL}$.
\end{proof}

\begin{lemma}  \label{lem:re5}
The modules $k_L^{\uparrow G}$ and $k_\whl^{\uparrow \whg}$ each 
have at most two indecomposable endotrivial direct summands. In 
each case, one of these summands is the trivial module $k$.  
\end{lemma}

\begin{proof}
First notice that, because $L \subseteq \whl$,  we have that
$(k_\whl^{\uparrow \whg})_{\downarrow G}$ is isomorphic to a direct
summand of $k_L^{\uparrow G}$.
So the result for $k_\whl^{\uparrow \whg}$ follows from 
the result for $k_L^{\uparrow G}$. 
Because $N = N_G(L)$ contains the normalizer of the 
Sylow subgroup $S$, we have that any endotrivial $kG$-module
is the Green correspondent of an endotrivial $kN$-module,
By transitivity of induction, $k_L^{\uparrow G} =
(k_L^{\uparrow N})^{\uparrow G}$, and there are no
more endotrivial $kG$-direct summands of $k_L^{\uparrow G}$ 
than there are endotrivial $kN$-direct summands of $k_L^{\uparrow N}
\cong k\CS_r$. The lemma is a consequence of the fact that
there are exactly two one dimensional direct summands 
of $k_L^{\uparrow G}$, namely $k$ and the sign module.
\end{proof}

From this point we use the notation. Let $\whcl = \whl((r-1)e,e)
= \GL((r-1)e,q) \times \GL(e,q)$ and let $\CL = \whcl \cap G$.

\begin{lemma} \label{lem:re6}
With the above notation, we have that 
\[
k_{\CL}^{\uparrow G} \quad \cong \quad 
(k_\whcl^{\uparrow \whg})_{\downarrow G}
\]
\end{lemma}

\begin{proof}
From the Mackey formula, we have that 
\[
(k_\whcl^{\uparrow \whg})_{\downarrow G} \ = \ 
\bigoplus_{x\in [G\backslash\whg/\whcl]} 
x \otimes k_{G \cap \whcl}^{\uparrow G} 
\cong \ k_\CL^{\uparrow G}
\]
since $G\whcl = \whg$ and there is only one $G$-$\whcl$ double 
coset which is represented by $x = 1$.
\end{proof}

Now we invoke the theory of Young modules. We have that 
\[
k_\whcl^{\uparrow \whg} \ \cong \ Y_{[(r-1)e,e]} \oplus 
\ \bigoplus_{j = 1}^{e-1} (Y_{[n-j,j]})^{a_j} \ \oplus \ Y_{[n]},
\]
for some multiplicities $a_j$. 
Note that $Y_{[n]} = k$ has multiplicity~$1$ by Frobenius 
reciprocity. Hence we only need to deal with 
the other summands. 

\begin{lemma} \label{lem:re7}
If $Y = (Y_{[(r-1)e,e]})_{\downarrow G}$ 
has an endotrivial direct summand, then 
it has only one. In that case $Y$ is indecomposable and 
$Y_{[(r-1)e,e]}$ is endotrivial as a $k\whg$-module.
\end{lemma}

\begin{proof}
The fact that there is at most one endotrivial direct summand of
$Y$ is a consequence of Lemma \ref{lem:re5}. By Clifford theory, 
the module $Y$ is a direct sum of conjugate $kG$-modules. If one 
of the summands is an endotrivial module then they all 
must be endotrivial. Since, by Lemma \ref{lem:re5}
there can be only one, $Y$ must be indecomposable and endotrivial. 
This means that it is also endotrivial as a $k\whg$-module. 
\end{proof}

Let $A$ be an invertible matrix of order $p$ in $\SL(e,q)$,
and let $X$ be the diagonal block matrix in $\SL(re, q)$ with 
$r$ diagonal blocks, each equal to $A$. We assume that the Sylow 
$p$-subgroup $S$ is chosen so that $X \in S$. Let $U = \langle X \rangle$
be the subgroup generated by $X$ in $G$. 

\begin{lemma}\label{lem:re8}
For all $1 \leq j < e$, the restriction of the Young module 
$Y_{[n-j,j]}$ to $U$ is a free module. Hence, $Y_{[n-j,j]}$ and its 
restriction to $G$ have no endotrivial direct summands. 
\end{lemma}

\begin{proof}
Let $\wtl = \whl(n-j, j) \cong \GL(n-j,q) \times \GL(j,q)$. 
We claim that for any $x \in \whg$, $U \cap x\wtl x^{-1} = \{1\}$.
Otherwise the matrix $x^{-1}Xx$ would have 
the form of diagonal blocks $X_1 \in \GL(n-j,q)$ and $X_2 \in 
\GL(j,q)$. However, the order of $\GL(j,q)$ is not divisible by $p$
and hence, $X_2$ is the identity matrix. This is a contradiction
because $X$ was chosen to have no eigenvalues equal to 1. This 
proves the claim. The lemma now follows by an easy argument using
the Mackey formula. 
\end{proof}

\begin{lemma}\label{lem:re9}
The restriction of the Young module $Y = Y_{[(r-1)e,e]}$ to $U$ 
has at least two trivial direct summands. That is, 
$Y_{\downarrow U} \cong k^a \oplus kU^b$, for some multiplicities
$a$ and $b$ where $a > 1$. 
\end{lemma}

\begin{proof}
Consider the Mackey formula
\[ 
(k_\whcl^{\uparrow \whg})_{\downarrow U} \quad \cong \quad 
\sum_{x\in [U\backslash \whg/\whcl]} x \otimes k_{U \cap \whcl^x}^{\uparrow U}
\]
Clearly, the multiplicity of $k_U$ as a direct summand is 
precisely the number of $U$-$\whcl$ double cosets represented
by elements in $N_G(U)$. Note that if $x \in N_G(U)$ then
$Ux\whcl = x\whcl$. Now $N_G(U)$ contains a subgroup $\Sigma
\cong \CS_r$. This subgroup is in the normalizer of the Levi
subgroup $\whl(e,e, \dots, e)$ of block diagonal $e\times e$
invertible matrices, where the conjugation action of $\Sigma$
permutes the blocks. The intersection of $\Sigma$ with $\whcl$
is isomorphic to $\CS_{r-1}$. Thus we have that $\whcl$ has 
at least $\vert \Sigma \vert/\vert \Sigma \cap \whcl \vert 
= r$ distinct left 
cosets represented by element of $N_G(U)$. Hence, 
the multiplicity of $k_U$ as a direct summand of 
$(k_\whcl^{\uparrow \whg})_{\downarrow U}$ is at least 
$r \geq 3$.

By Lemma \ref{lem:re8}, the 
multiplicity of $k_U$ as a direct summand of 
$(Y_{[n-j,j]})_{\downarrow U}$
is zero for $1 \leq j < e$. The multiplicity $k_U$ in 
$(Y_{[n]})_{\downarrow U} \cong k$ is exactly one. So the lemma 
follows from the decomposition of $k_\whcl^{\uparrow \whg}$ as 
a sum of Young modules. 
\end{proof}

The above results are sufficient to prove the main theorem of
this section.

\begin{thm} \label{thm:re}
Suppose that $n = re$ for $e > 1$ and $r\geq2$,
and that $G = \SL(n,q)$. Then $TT(G) = \{0\}$.
\end{thm}

\begin{proof}
First, for $r=2$ the claim is proved in Theorem~\ref{thm:2e}. Now 
suppose that $r>2$ and that $M$ is an indecomposable endotrivial
$kG$-module with trivial Sylow restriction and $M \not\cong k$. Then by
Lemma \ref{lem:re4}, $M$ is a direct summand of $k_\CL^{\uparrow G}$.
Then by Lemmas \ref{lem:re6}, \ref{lem:re7} and \ref{lem:re8}, $M$ is 
isomorphic to the restriction of $Y_{[(r-1)e,e]}$ to $G$ and 
$Y_{[(r-1)e,e]}$ must be an endotrivial $k\whg$-module. However,
by Lemma \ref{lem:re9}, this is not the case. 
\end{proof}


\section{Endotrivial modules for $\SL(re+f, q)$,
$2 \leq r < p$, $1 \leq f <e$}\label{sec:case-re+s}
In this section, we present the final case of the endotrivial 
modules of $\SL(n,q)$ under the assumption that $S$ is abelian, namely,
the case in which $n= re+f$ with $r<p$ and $f \geq1$. In the proof we again
use Balmer's method for computing the kernel of the restriction
map. Indeed, we use the method twice in different contexts.  On one
occasion, we show that $A(G,H) = \{1\}$ where $H$ is the commutator 
subgroup of the maximal parabolic subgroup $P(n-1,1)$. The proof
is by induction on $f$, the result being known for the 
case that $f=0$ by the main theorem of the last section.
The theorem that we prove is slightly more general than what
has been stated above. We hope this generality will be useful
in later work. 

Because we assume that $f > 1$, the power of $p$
dividing $\vert \SL(n-1, q)\vert $ is the same 
as that dividing $\vert \SL(n,q)\vert$, 
so that the Levi subgroup $L = L(n-1,1)$ contains a Sylow 
$p$-subgroup of $G$. For notation, let $P = P(n-1,1)$, the maximal parabolic 
subgroup of $G = \SL(n,q)$ consisting of upper block 
triangular invertible matrices with diagonal blocks of size
$(n-1) \times (n-1)$ and $1 \times 1$, having the product
of their determinants equal to $1$. So an element of 
$P$ has the form 
\[ 
X \quad := \quad \begin{bmatrix} A & v \\ 0 & \zeta \end{bmatrix}
\]
where $A \in \GL(n-1,q)$, $\zeta \in k^\times$ and $v \in \bV$, 
the natural module for $\GL(n-1,q)$ (the $k$-vector space of 
dimension $n-1$ with elements written as 
column vectors). Also, $\Det(A) = \zeta^{-1}$. 
Let $H = [P, P]$, the commutator subgroup of $P$.
It is not difficult to prove that $H$ consists of all elements of
the above form $X$ with the additional stipulations that 
$A \in \SL(n-1,q)$ and $\zeta = 1$. The element $v$ can still be
any element of $\bV$. 

Let $V \subseteq H$ be the set of all 
elements having the above form with $A = I_{n-1}$ the identity 
matrix and $\zeta = 1$. So the map $\varphi:\bV \to V$ that sends 
$v \in \bV$ to $\left [ \begin{smallmatrix} I_{n-1} &v \\ 0 & 1 
\end{smallmatrix} \right]$ is an isomorphism of abelian groups. 
Indeed, if $\vartheta: \SL(n-1, q) \to H$ is the injection that 
sends $A$ to 
$\left[ \begin{smallmatrix} A & 0 \\ 0 & 1 
\end{smallmatrix} \right]$, then we have that 
$\varphi(Av) = \vartheta(A) \varphi(v) \vartheta(A)^{-1}$. This 
fact is very useful in one of the proofs that follow.  
Finally, let $L \cong \SL(n-1,q)$ denote the 
image of $\vartheta$. Thus, $V$ is a normal subgroup of $H$ 
with quotient $H/V \cong L$. 

\begin{lemma} \label{lem:LLcosets}
There are exactly two $L$-$L$-double cosets in $H$, and they
are represented by
~$1$ and the element $w = \varphi((0,0, \dots, 1))$. 
\end{lemma}

\begin{proof}
Let $\hat{w} = (0, \dots, 0,1) \in \bV$ a preimage of $w$ under $\varphi$.
Because $\bV$ is an irreducible module and $\SL(n-1,q)$ acts 
transitively on $\bV$, for any nonzero $v \in V$ with preimage 
$\hat{v}$ in $\bV$, we have that there is some $A$ in $\SL(n-1,q)$
such that $A\hat{v} = \hat{w}$ and $\vartheta(A)v = w$. On double 
cosets this means that $LvL = LwL$ for all $v \neq 1$ in $V$. 
Since $H = VL$, we have proved the lemma. 
\end{proof}

\begin{prop} \label{prop:tth}
The restriction map $T(H) \to T(L)$ is 
injective. Hence, if $TT(L) = \{0\}$, then $TT(H) = \{0\}$.
\end{prop}

\begin{proof}
We have an element of order $p$ of the form 
$U = \left[ \begin{smallmatrix} u & 0 \\ 0 & I_{n-e}
\end{smallmatrix} \right]$ in $H$,
where $u$ is an element of order $p$ in $\SL(e,q)$. The 
centralizer contains the subgroup $J$ of all matrices of 
the form $\left[ \begin{smallmatrix} I_{e} & 0 \\ 0 & A 
\end{smallmatrix} \right]$ for $A$ in $\SL(n-e,k)$. The 
element $w$ of the previous lemma is in $J$ and in the 
commutator subgroup of the centralizer of $U$.
So by Lemma \ref{lem:LLcosets} and Proposition 
\ref{prop:balmerapp1}, the group $A(H,L)$ of weak $L$-homomorphisms is
trivial. So by Theorem \ref{thm:balmer}, the restriction $T(H) \to T(L)$ is
injective.  
\end{proof}

\begin{lemma} \label{lem:HPcosets}
The left and right cosets of $H$ in $P$ are represented 
by elements of the form 
\[ 
a_\zeta \quad = \quad \begin{bmatrix} I_{n-2} & 0 \\ 0 & Y
\end{bmatrix} \qquad \text{where} \qquad 
Y = \begin{bmatrix} \zeta & 0 \\ 0 & \zeta^{-1} \end{bmatrix},
\] 
for all $\zeta \in \bfq^\times$. 
\end{lemma}

\begin{proof}
This is a direct calculation. 
\end{proof}

\begin{lemma} \label{lem:PPcosets}
There are exactly two $P$-$P$ double cosets in G, that are 
represented by the elements ~1 and 
\[ 
b \quad = \quad \begin{bmatrix} I_{n-2} & 0 \\ 0 & A
\end{bmatrix} \qquad \text{where} \qquad 
A = \begin{bmatrix} 0 & 1 \\ -1 & 0 \end{bmatrix},
\] 
\end{lemma}

\begin{proof}
We obtain a complete set of double coset representatives from 
\cite[Prop. 2.8.1(iii)]{Ca2} as follows. Let 
$\Delta=\{\alpha_{1},\alpha_{2},\dots,\alpha_{n-1}\}$ be the simple
 roots. Set $P_{J}=P$ where 
$\Delta_{J}=\{\alpha_{1},\alpha_{2},\dots, \alpha_{n-2}\}\subseteq \Delta$.
Let $D_{J}=\{w\in W: w(\Delta_{J})\subseteq \Phi^{+}\}$ ($\Phi^{+}$ 
is the set of positive roots). Let $s_{j}=s_{\alpha_{j}}$ 
for $j=1,2,\dots,n-1$. In this case
$$
D_{J}=\{1,\ s_{n-1}, \ s_{n-2}s_{n-1},\ \dots, \ s_{1}s_{2}\cdots s_{n-1}\}.
$$
Then a set of double coset representatives are given by 
$D_{J}\cap D_{J}^{-1}=\{1,\ s_{n-1}\}$ (cf. \cite[2.7 Definition]{Ca2}).  
The matrix $b$ represents $s_{n-1}\in W$.
\end{proof}

\begin{lemma} \label{lem:HHcosets}
Every $H$-$H$ double coset in $G$ is represented by an 
element of the form 
\[
X \quad = \quad \begin{bmatrix} I_{n-2} & 0 \\ 0 & W
\end{bmatrix} \qquad \text{where} \qquad 
W \in \SL(2,q).
\] 
Any such $X$ is in the commutator subgroup of the 
subgroup $U$ in the proof of Proposition \ref{prop:tth}
\end{lemma}

\begin{proof}
Every element of $x \in G$ has either the form $x = u$
or $x = ubv$ for $u,  v \in P$ by Lemma \ref{lem:PPcosets}.
So by Lemma \ref{lem:HPcosets} we can write $u = ha_\zeta$ 
and $v = a_\eta h^\prime$ for some $h, h^\prime \in H$ and
some $\zeta, \eta \in \bfq^\times$. Thus the double coset 
representatives can be taken to have the form $a_\zeta$ 
for some $\zeta$, or $a_\zeta b a_\eta$ for some $\zeta$ and 
some $\eta$. This proves the lemma as the last statement is obvious. 
\end{proof}

At this point we can prove the main theorem of the section.
We regard $\SL(n-1,q)$ as the subgroup of $\SL(n,q)$ consisting
of block diagonal matrices of sizes $(n-1) \times (n-1)$ and 
$1 \times 1$. This is exactly as $L \subseteq G$ in the notation
above.

\begin{thm} \label{thm:restrSL}
Suppose that $\SL(n-1,q)$ contains a Sylow $p$-subgroup of 
$G$ and the Sylow $p$-subgroup has $p$-rank at least two. 
Then the group $A(\SL(n,q), \SL(n-1,q))$ of weak
$\SL(n-1,q)$-homomorphisms of $\SL(n,q)$ is trivial, and so the
restriction map $T(\SL(n,q)) \to T(\SL(n-1,q))$ is injective. 
\end{thm}

\begin{proof}
The hypothesis of the Sylow $p$-subgroup of $G$ assures us
that $n > 2e.$  Suppose that $H = \SL(n-1,q) \subseteq G = \SL(n,q)$.
Then we argue as in the proof of Proposition \ref{prop:tth} that by
Lemma \ref{lem:HHcosets}, a complete set of representatives
of $H$-$H$ double cosets can be chosen in the commutator subgroup of the
centralizer of a nontrivial $p$-subgroup. Thus by Proposition
\ref{prop:balmerapp1} and Theorem \ref{thm:balmer}, we conclude that
$A(G,H) = \{1\}$ and the restriction map $T(G) \to T(H)$ is
injective. So the composition of restriction maps $T(G) \to T(L)$ is
injective by Proposition \ref{prop:tth}. 
\end{proof}

\begin{cor} \label{cor:restrict-re}
Suppose that $n = re+f$ for $2 \leq r < p$ and $0 \leq f <e$.
Then $TT(\SL(n, q)) = \{0\}.$
\end{cor}

\begin{proof}
According to Theorem \ref{thm:re} we have $TT(\SL(re, q)) = \{0\}$.
Now one can apply induction on $f$ using Theorem \ref{thm:restrSL}
to complete the proof of the corollary.
\end{proof}


\section{Proofs of Theorems 1.1 and 1.2}
In this section we prove the theorems 
of the introduction. The proof in the
case that the Sylow $p$-subgroup is cyclic is a direct 
calculation based on known results. The proof in the case that the
Sylow $p$-subgroup is abelian and has $p$-rank at least $2$ is a
compilation of results from prior sections. All of this is extended to
finite groups of Lie type with underlying root system of type $A_{n}$.

\begin{proof}[Proof of Theorem \ref{thm:noncyclic}]
We have proved that $T(SL(n,q)) = \bZ$ whenever the Sylow $p$-subgroup
is abelian and has $p$-rank at least $2$ in Theorems \ref{thm:torfree},
\ref{thm:2e}, \ref{thm:re} and \ref{cor:restrict-re}. 
If $\SL(n,q) \subseteq G \subseteq \GL(n,q)$, then 
a proof that $T(G) \cong \bZ \oplus X(G)$ can be constructed using
Clifford theory. That is, since any indecomposable endotrivial 
$kG$-module, restricts to a sum of conjugate $k\SL(n,q)$-modules, the
restriction must have only one summand and must have dimension
one if the endotrivial module has trivial Sylow restriction.

If $Z \subseteq Z(G)$, then we note that any endotrivial 
$k(G/Z)$-module must inflate to an endotrivial $kG$-module. 
\end{proof}

\begin{proof}[Proof of Theorem \ref{thm:cyclic}]
Observe that if $p$ divides $d$
or if $p = 2$,  then the 
assumption that a Sylow $p$-subgroup of $G$ is cyclic, requires
that  $n=1$. So in this case $G \cong D$ is abelian and the 
structure of $T(G/Z)$ is easily computed. 

Recall from Theorem \ref{thm:cyclic2} that $T(G/Z) \cong
T(\widetilde{N})$ where $\widetilde{N} = N_{G/Z}(\widetilde{S}).$
Here, $\widetilde{S}$ is the unique subgroup of order $p$ in 
a Sylow $p$-subgroup of $G/Z$. In this case, the normalizer of 
$\widetilde{S}$ coincides with the normalizer of $\whs$, the
Sylow $p$-subgroup of $G/Z$. Note that 
$\whs \cong ZS/Z$ and its normalizer is $\whn \cong N/Z$.  
The displayed sequence (\ref{eq:sequence}) follows from the sequence in 
Theorem \ref{thm:cyclic2} and the fact that $X(\whn) 
\cong \whn/[\whn,\whn]$ where $\whs\subseteq[\whn,\whn]$ and
\[
\whn/[\whn, \whn] \ \cong \ 
(N/Z)/[N/Z, N/Z] \ \cong \  (N/Z)/(Z[N,N]/Z) \ \cong 
\ N/(Z[N,N]). 
\]

Now assume that $p>2$ divides $q-1$, but not $d$. Then 
$n = 2$, and $N$ is generated by elements 
\[
x = \begin{bmatrix} 0 & 1 \\ -1& 0 \end{bmatrix}, \quad 
u_a = \begin{bmatrix} a & 0 \\ 0 & 1 \end{bmatrix}, \quad \text{and}
\quad v_b = \begin{bmatrix} b & 0 \\ 0 & b^{-1} \end{bmatrix}
\]
for $a \in D$, and $b \in \bF_q^{\times}$. The Sylow $p$-subgroup
is generated by $u_\zeta$ where $\zeta \in \bF_q^\times$ generates a 
Sylow $p$-subgroup of $\bfq^\times$. 

Let $T$ denote the torus in $\SL(n,q)$ 
that consists of all $v_b$ for $b \in \bfq^\times$. 
It is easily seen that the commutator subgroup $[N,N]$ of
$N$ is generated by all $x^{-1}u_{a^{-1}}xu_a = v_a$ and 
$x^{-1}v_{b^{-1}}xv_b = v_{b^2}$. It follows that $[N, N] =T$ if and 
only if $D$ contains a generator for the Sylow $2$-subgroup
of $\bfq^\times$. This happens if and only if~$2$ does not 
divide $(q-1)/d$.  Otherwise, the index of $[N,N]$ in $T$ is $2$. 

In the first case when $[N,N] = T$, we have that 
$N/[N,N] \cong D \times C_2$ where the factors are generated
by the classes (modulo $[N,N]$) of $u_a$ for $a$ a generator of $D$ and
$x$. Then $T(G) \cong T(N) \cong \bZ/d\bZ \oplus \bZ/4\bZ$ because the element
$x$ acts by inverting the elements of $S$ and hence the class of
$\Omega(k)$ in $T(G)$ has order~$4$.
That is, $\Omega^2(k)$ is a $kG$-module of dimension one on 
which $x$ acts by multiplication by $-1$. 

In the second case, when the index of $[N,N]$ in $T$ is $2$, we have that 
$N/[N,N] \cong D \times C_2 \times C_2$ 
where the factors are generated
by the classes of $u_a$, for a generator $a$ of $D$, $x$, and $v_b$, for
a generator $b$ of the Sylow $2$-subgroup of $T$. Then 
$$
T(G) \cong T(N) \cong \bZ/d\bZ \ \oplus \\bZ/4\bZ \ \oplus \ \bZ/2\bZ,
$$
for the same reason as above. 

We invoke the notation in Section \ref{sec:GLinfo}. Suppose that $e>1$ and assume 
first that $f = 0$. By Lemma \ref{lem:glebasics}, parts (a) and (c), we have
that $C_G(S) = \langle w^m \rangle = G \cap C_\whg(S)$
for $\whg = \GL(e,q)$ and $m= (q-1)/d$. Thus by part (b), $N$ is
generated by $w^m$ and $g$. We can see that $[N,N]$ is generated by
$gw^mg^{-1}w^{-m} = w^{m(q-1)}$. So the index of $[N,N]$ in $C_G(S)$ is
$\ell=\gcd\big(m(q-1),q^e-1\big)/m$. It follows that $N/[N,N] \cong
C_\ell \times C_e$, and that $T(G)$ has the asserted form.  

Finally suppose that $e>1$ and $f > 0$. Let $\theta: \GL(e, q)
\to \SL(n,q) \subseteq G$ be the homomorphism that takes an element $x$ 
to the block matrix with $x$ in the upper left corner. The lower 
right corner of $\theta(x)$ is the diagonal $f \times f$ matrix
with $Det(x)^{-1}$ in the upper left corner, the other diagonal 
entries being equal to 1.  The normalizer $N_G(S)$ consists of 
all elements of the form $\left[ \begin{smallmatrix} A&0\\0&B \
\end{smallmatrix}\right]$ such that $A \in N_{\GL(e,q)}(\langle u \rangle)$,
$B \in \GL(f,q)$ and $\Det(A)\Det(B) \in D = \Det(G)$. 
Then, in almost all cases, we see that
$[N,N]$ is generated by $\theta(w^{q-1})$ and all elements (as above)
with $A = I_e$ and $B \in \SL(f,q)$.
The exception to the above occurs when $f = 2$ and $q=2$, because
in that case the commutator subgroup of $\GL(f,q)$ is not equal
to $\SL(f,q)$, but is a subgroup of index $2$. 
As a consequence we have (in all but the noted exceptional case) that 
$N/[N,N] \cong C_e \times C_{q-1} \times C_d$,
where the first two factors are generated by the classes modulo
$[N,N]$ of $\theta(g)$ and $\theta(w)$. The third factor is generated
by the class of an element given by the diagonal matrix that has 1
in every diagonal entry except the $(e+1)\times(e+1)$ entry which is 
a generator for the group $D \subseteq \bfq^{\times}$.
Then $T(G)$ has the asserted form. 
\end{proof}

\end{document}